\documentclass[11pt]{article}

\usepackage[utf8]{inputenc}

\usepackage{calrsfs}
\usepackage{musicography}
\usepackage[colorlinks]{hyperref}

\usepackage{tikz}

\usepackage{caption}

\usepackage{url}

\usepackage{amssymb}
\usepackage{mathrsfs}
\usepackage{frcursive}
\usepackage{color}
\usepackage{stmaryrd}
\usepackage{cancel}
\usepackage{scalerel,stackengine}
\newcommand\reallywidehat[1]{%
\savestack{\tmpbox}{\stretchto{%
  \scaleto{%
    \scalerel*[\widthof{\ensuremath{#1}}]{\kern-.6pt\bigwedge\kern-.6pt}%
    {\rule[-\textheight/2]{1ex}{\textheight}}
  }{\textheight}%
}{0.5ex}}%
\stackon[1pt]{#1}{\tmpbox}%
}
\usepackage{csquotes}
\usepackage{etoolbox}

\AtBeginEnvironment{thebibliography}{\small}
\usepackage{titlesec}
\usepackage{lipsum}

\usepackage{mathrsfs}

\usepackage{enumitem}

\usepackage{amsthm,amsmath,amsfonts,amssymb,hyperref,cases,color}

\usepackage[capitalise]{cleveref}

\usepackage{vmargin}

\usepackage{amsthm}

\usepackage{hyperref}

\setmarginsrb{2.4cm}{1.2cm}{2.4cm}{2.4cm}{1cm}{1cm}{2cm}{2cm}

\usepackage{tikz-cd}

\usepackage{slashed}

\usepackage[T1]{fontenc}

\usetikzlibrary{decorations.markings}

\numberwithin{equation}{section}

\usepackage{ amssymb }

\usepackage{ dsfont }

\usepackage{calrsfs}
\DeclareMathAlphabet{\pazocal}{OMS}{zplm}{m}{n}

\newcommand{\C}{\mathbb{C}}
\newcommand{\R}{\mathbb{R}}
\newcommand{\N}{\mathbb{N}}

\newcommand{\Z}{\mathbb{Z}}

\newcommand{\Pa}{(P -i\epsilon)^{-\alpha}}

\newcommand{\resP}{(P - z)^{-1}}

\newcommand{\U}{\pazocal{U}}
\newcommand{\Dd}{\pazocal{D}}

\newcommand{\tr}{\mathrm{tr}}

\newcommand{\loc}{\mathrm{loc}}

\newcommand{\wf}{\mathrm{WF}}

\newcommand{\Id}{\mathrm{Id}}

\newcommand{\overbar}[1]{\mkern 0.5mu\overline{\mkern-4.0mu#1\mkern-1.0mu}\mkern 1.0mu}

\makeatletter
\newsavebox\myboxA
\newsavebox\myboxB
\newlength\mylenA

\newcommand*\xoverline[2][0.75]{%
    \sbox{\myboxA}{$\m@th#2$}%
    \setbox\myboxB\rholl
    \ht\myboxB=\ht\myboxA%
    \dp\myboxB=\dp\myboxA%
    \wd\myboxB=#1\wd\myboxA
    \sbox\myboxB{$\m@th\overline{\copy\myboxB}$}
    \setlength\mylenA{\the\wd\myboxA}
    \addtolength\mylenA{-\the\wd\myboxB}%
    \ifdim\wd\myboxB<\wd\myboxA%
       \rlap{\hskip 0.5\mylenA\usebox\myboxB}{\usebox\myboxA}%
    \else
        \hskip -0.5\mylenA\rlap{\usebox\myboxA}{\hskip 0.5\mylenA\usebox\myboxB}%
    \fi}
\makeatother

\newcommand{\Mb}{\overbar{M}}

\renewcommand{\epsilon}{\varepsilon}
\renewcommand{\Re}{\mathrm{Re}}

\usepackage{etoolbox}

\makeatletter
\apptocmd{\thebibliography}{%
  \setlength{\itemsep}{1pt}%
  \setlength{\parskip}{1pt}%
}{}{}
\makeatother

\usepackage{setspace}
\theoremstyle{plain}
\newtheorem{theorem}{Theorem}[section]
\newtheorem{lemma}[theorem]{Lemma}
\newtheorem{hypothesis}[theorem]{Hypothesis}
\newtheorem{example}[theorem]{Example}
\newtheorem{proposition}[theorem]{Proposition}

\newtheorem{definition}[theorem]{Definition}
\newtheorem{remark}[theorem]{Remark}

\usepackage[labelfont=it]{caption}

\usepackage{titlesec}

\makeatletter
\renewenvironment{proof}[1][\proofname]{\par
  \pushQED{\qed}%
  \normalfont \topsep6\p@\@plus6\p@\relax
  \trivlist
  \item[\hskip\labelsep
        \bfseries #1.] 
}{%
  \popQED\endtrivlist\@endpefalse
}
\makeatother
\usepackage{fancyhdr}
\usepackage{xcolor}
\usepackage{hyperref}

\definecolor{mygreen}{RGB}{0,120,70}

\usepackage{xcolor}

\definecolor{bibgreen}{RGB}{0,120,70}

\usepackage{hyperref}

\definecolor{citeblue}{RGB}{85,125,185}   
\definecolor{urlgreen}{RGB}{75,135,90}
\definecolor{linkblue}{RGB}{45,100,180}
\definecolor{darkred}{RGB}{185,60,65}

\hypersetup{
    colorlinks=true,
    linkcolor=darkred,
    citecolor=linkblue,
    urlcolor=urlgreen
}

\fancypagestyle{plain}{
  \fancyhf{}
  \fancyfoot[C]{\raisebox{10mm}{\thepage}}

}

\usepackage[utf8]{inputenc}
\usepackage[T1]{fontenc}
\usepackage[english]{babel}
\usepackage{titlesec} 

\makeatletter
\renewcommand{\@seccntformat}[1]{%
  \csname the#1\endcsname.\quad
}
\makeatother

\titleformat{\section}[block]
  {\normalfont\large\bfseries \scshape \centering}
  {\thesection.}
  {0.3em}
  {}

\titleformat{\subsection}[runin]
  {\normalfont\normalsize\bfseries}
  {\thesubsection.}
  {0.35em}
  {}
  [.]

\titlespacing*{\subsection}
  {0pt}
  {1.25ex plus .2ex minus .2ex}
  {0.5em}

\titleformat{\subsubsection}[runin]
  {\normalfont\normalsize}
  {\thesubsubsection.}
  {0.35em}
  {\itshape}
  [.]
\titlespacing*{\subsubsection}
  {0pt}
  {1.25ex plus .2ex minus .2ex}
  {0.5em}

\usepackage{titling}
\posttitle{\par\end{center}\vspace{-0.4em}}

\usepackage{etoolbox}
\AtBeginEnvironment{abstract}{\small}

\title{\large\bfseries \scshape 
THE PSEUDO-RIEMANNIAN SPECTRAL ZETA FUNCTION DENSITY ON SCATTERING MANIFOLDS
}

\author{\scshape Ashkan Sadat Kyaee \\[0em]}

\date{}

\usepackage{fancyhdr}

\usepackage{titling}

\begin{document}

\vspace{2em}
  
\maketitle
\vspace{-4em}

\begin{center}
\begin{minipage}{0.92\textwidth}

\fontsize{10}{12}\selectfont \textbf{\scshape Abstract.} We study spectral zeta densities associated with complex powers of principally scalar Laplace-type operators on pseudo-Riemannian scattering spaces of mixed signature and arbitrary dimension. The construction extends, beyond the even-dimensional Lorentzian case in \cite{Dang:2020jch}, the local analysis of Feynman complex powers developed by Dang--Wrochna. Namely, for $\varepsilon > 0$, we consider complex powers \((P\pm i\varepsilon)^{-\alpha}\) and  prove that the diagonal trace densities admit meromorphic continuation. We relate their poles to local geometric invariants, in particular to the pseudo-Riemannian scalar curvature. Our proof combines uniform Feynman resolvent estimates with a direct analysis of the Hadamard parametrix in pseudo-Riemannian signature. Beyond the Lorentzian case, the use of causal convexity is replaced by a geometric no-return argument.

\end{minipage}
\end{center}

\vspace{1.5 em}

\section{introduction}

\subsection{Introduction and main results}

 Consider a compact Riemannian manifold $(M,g)$ of dimension $n$, and let $\Delta_g$ be the Laplace--Beltrami operator. Minakshisundaram--Pleijel \cite{Minakshisundaram_Pleijel_1949} and Seeley \cite{Seeley:1967ea} proved that the trace density of $(- \Delta_g)^{- \alpha}$, defined as the on-diagonal restriction of the Schwartz kernel $(- \Delta_g)^{- \alpha}(x,x')$, exists for all $\alpha \in \C$ such that $\Re(\alpha) > \frac{n}{2}$ and extends to a density-valued meromorphic
function. Kastler \cite{6Kastler:1993zj} and Kalau--Walze \cite{7Kalau_1995} expressed its residues in terms of local geometric invariants, such as the scalar curvature. Namely,
for $n \geq 4$, 
\begin{equation*} 
    \operatorname*{res}_{\alpha=\frac n2-1} (- \Delta_g)^{- \alpha}(x,x) = \frac{R_g(x)}{6(4\pi)^{\frac{n}{2}} \Gamma\left(\frac{n}{2} -1 \right)}, 
\end{equation*}
where $R_g(x)$ is the scalar curvature of $(M,g)$ at $x$. In dimension $n = 4$, this formula made precise Connes' observation that the
Wodzicki residue of the inverse square of the Dirac operator is proportional to the Einstein--Hilbert action of general relativity.  In particular, this means that gravity could be derived from a \emph{spectral action}, at least in Riemannian signature; see \cite{Chamseddine_1997}.

The generalization to Lorentzian manifolds, and, a fortiori, to pseudo-Riemannian manifolds, is highly nontrivial. While formal heat kernel coefficients for the wave operator on a Lorentzian manifold can be defined through transport equations similar to those in the Riemannian setting and interpreted using a Lorentzian Hadamard parametrix, their precise connection with global spectral-theoretic objects presents many challenges. Amongst others, the main difficulty is without any doubt the non-ellipticity of the wave operator. A major advance was the construction of the complex powers of the wave operator on \emph{non-trapping Lorentzian scattering spaces} by Dang and Wrochna
\cite{Dang:2020jch}, where the metric and the Hamilton flow associated to the principal symbol of the wave operator are assumed to behave asymptotically in a certain way.  Their work relies on the essential self-adjointness of
the wave operator, due to Vasy
\cite{Vasy2017EssentialSO}, together with Feynman resolvent estimates and a
uniform Hadamard parametrix. In even dimension, Dang and Wrochna proved that complex powers of the wave operator are related to the underlying spacetime geometry in a fashion largely analogous to what is known in the Riemannian setting, which confirms en passant the actual physical significance of the spectral action principle; see the scalar-curvature residue in 
\cite[Prop.~8.5]{Dang:2020jch}. These Lorentzian residues coincide with the dynamical residues defined in \cite{DangWrochna2022}, which generalize the Guillemin–Wodzicki residue density of pseudodifferential operators.

The present work addresses the problem of extending these results to general pseudo-Riemannian signature, as well as to odd-dimensional manifolds. We formulate
our results for principally scalar Laplace-type operators acting on a Hermitian vector bundle \(E\to M\). The main new difficulty in higher signature is not the local Hadamard
construction itself, which remains available in convex normal neighborhoods,
but the control of the error created by cutting off the parametrix. In the
Lorentzian case, causal convexity prevents null bicharacteristics from leaving
the local region and returning to the diagonal. In general pseudo-Riemannian
signature, there is no causal order with this property. We nevertheless show that the complex powers admit a well-defined on-diagonal restriction and compute the corresponding residues. Our main result is the following theorem.

 \begin{theorem} [cf.~Theorem $\ref{thmexpanresol}$ and Proposition \ref{propthmfinal}] \label{mainthemintro} Let $(M,g)$ be a pseudo-Riemannian scattering space of arbitrary dimension $n$ and signature $(\kappa,n - \kappa)$, and let $E$ be a Hermitian bundle. Let $P \in \mathrm{Diff}^{2,0}_{\mathrm{sc}}(M;E)$ be of Laplace-type and assume it is non-trapping in the sense of Hypothesis \ref{theH}. Then, for all $\epsilon > 0$, the Schwartz kernel of $(P {\pm} i\epsilon)^{-\alpha}$ has for $\Re(\alpha) > \frac{n}{2}$ a well-defined on-diagonal restriction $\iota_\Delta^*(P\pm i\varepsilon)^{-\alpha} \in \pazocal{D}'\left(M;\mathrm{End}(E)\right),$ which extends as a meromorphic function of $\alpha \in \C$ with poles in $$\pazocal{P}_n := \left\{ \frac{n}{2} - k \ \vert \ k \in \N_0, \ \tfrac{n}{2} - k \notin - \N_0 \right\}.$$ The residues are given by
\begin{equation*} 
    \lim_{\epsilon \to 0^+}
    \operatorname*{res}_{\alpha=\frac n2-k}
    (P \pm i\epsilon)^{-\alpha}(x,x)
    =
    \frac{(\mp i)^\kappa u_k(x,x)}
    {(4\pi)^{n/2}\Gamma\!\left(\frac n2-k\right)} \in \mathrm{End}(E_x),
\end{equation*}
for every $k \in \N_0$ such that $\tfrac{n}{2} - k \in \pazocal{P}_n$, where the $u_k$ are the Hadamard coefficients of $P$ in the convention of Section \ref{transeq}.
\end{theorem}
Theorem \ref{mainthemintro} can be considered as the pseudo-Riemannian version of the results of Kastler and Kalau--Walze, built upon the framework and techniques developed by Dang--Wrochna and Vasy in \cite{Dang:2020jch, dang2024diracoperatorslocalinvariants} and extended in arbitrary signature and dimension. In particular, for $n \geq 3$, Proposition~\ref{propscalarcurvgeneral} shows that the residue at \(\alpha=\frac n2-1\) is governed by the local invariant
$$
-\frac{R_g}{6}\operatorname{Id}_E-B,$$
making the scalar-curvature contribution explicit for general Laplace-type operators \(P=\Box^\nabla+B\).

As an application, we choose to present the odd-dimensional analogues of the results proven by Dang--Wrochna in \cite{Dang:2020jch}, summarized in the note \cite{dang2022lorentzianspectralzetafunctions}, for $P = \Box_g$ the wave operator.
\begin{theorem} [cf.~Theorem \ref{thmresidlorentz} and Proposition \ref{proprescurvlorentz}] Let $(M,g)$ be a globally hyperbolic, non-trapping Lorentzian scattering space and assume its dimension $n$ is odd. Then, for all $\epsilon > 0$, the Schwartz kernel of $(\Box_g \pm i\epsilon)^{-\alpha}$ has for $\Re(\alpha) > \frac{n}{2}$ a well-defined on-diagonal restriction $\iota_\Delta^*(\Box_g \pm i\varepsilon)^{-\alpha} \in \pazocal{D}'(M),$ which extends as a meromorphic function of $\alpha \in \C$ with poles in $\pazocal{P}_n$. In particular, for $n \geq 3$,
    $$ \lim_{\epsilon \to 0^+}  \operatorname*{res}_{\alpha=\frac n2-1}(\Box_g  \pm i\epsilon)^{-\alpha}(x,x) = \pm \frac{i R_g(x)}{6(4\pi)^{\frac{n}{2}} \Gamma\left(\frac{n}{2} -1 \right)}, $$
    where $R_g(x)$ is the scalar curvature of $(M,g)$ at $x$.
\end{theorem}

Finally, we introduce the \emph{scattering Dirac operator}  and extend the residue formula of \cite[Thm.~5.7]{dang2024diracoperatorslocalinvariants} to the complex powers of its square in arbitrary signature and dimension.

\begin{theorem} [cf.~Theorem \ref{thmdiracresidue}]
    Let $(M,g)$ be a pseudo-Riemannian scattering space of dimension $n$ and signature $(\kappa,n-\kappa)$. Consider a Clifford module $(E,\nabla^E)$, let $\slashed{D} \in \operatorname{Diff}_{\mathrm{sc}}^{1,0}(M;E)$ be a non-trapping Dirac operator in the sense of Definition \ref{Dnontrap} and put $ P=-\slashed{D}^{\,2}$. Then, for all $\epsilon > 0$, the Schwartz kernel of $(P {\pm} i\epsilon)^{-\alpha}$ has for $\Re(\alpha) > \frac{n}{2}$ a well-defined on-diagonal restriction $$ \iota_\Delta^*(P\pm i\varepsilon)^{-\alpha} \in \pazocal{D}'\left(M;\mathrm{End}(E)\right),$$ which extends as a meromorphic function of $\alpha \in \C$ with poles in $\pazocal{P}_n$. Furthermore, for $n \geq 3$,
\begin{equation} \label{psRspectral}
    \lim_{\varepsilon\to0^+}
\underset{\alpha=\frac{n}{2}-1}{\operatorname{res}}
    (P\pm i\varepsilon)^{-\alpha}(x,x)
= \frac{(\mp i )^{\kappa}R_g(x)}
     {12(4\pi)^{\frac n2}
      \Gamma\!\left(\frac n2-1\right)} \mathrm{Id}_{{E_x}} + \frac{(\mp i )^{\kappa}\mathbf{F}^E(x)}
     {(4\pi)^{\frac n2}
      \Gamma\!\left(\frac n2-1\right)},
\end{equation}
where $\mathbf{F}^E$ is the Clifford contraction of the twisting curvature $F^E$ of $E$.
\end{theorem}

\subsection{Main ideas of the proof} For a differential operator $P \in \mathrm{Diff}^{2,0}_{\mathrm{sc}}(M;E)$ of Laplace-type, the first key feature that we exploit is that, for $\mathrm{Im}(z) > 0$, the resolvent $(P-z)^{-1}$ is a \emph{Feynman inverse}. In other words, it behaves as the \textit{Feynman propagator} on pseudo-Euclidean space, which is the Fourier multiplier by $$(Q_\kappa(\xi) - i0)^{-1}, \qquad Q_\kappa(\xi)
    =
    -\xi_1^2-\cdots-\xi_\kappa^2
    +\xi_{\kappa+1}^2+\cdots+\xi_n^2. $$ This result is formulated in terms of a uniform operator Sobolev wavefront set $\wf'^{(s)}_{h(z)}$, and relies on the non-trapping structure of $(M,g)$. Consequently, the Schwartz kernel of the resolvent $(P-z)^{-1}$ can be approximated by the Feynman version of the \textit{uniform Hadamard parametrix} denoted by $H_N(z)$, expressed in terms of the elementary family of distributions $\mathbf{F}_{\alpha}(z)$ on a specific neighborhood of the diagonal $\Delta$ in $M \times M$. The approximation of the resolvent by the Hadamard parametrix will read 
    \begin{equation} \label{approxid}
        \resP = H_N(z) + E_{N,1}(z) + E_{N,2}(z),
    \end{equation}
    where $E_{N,1}(z)$ and $E_{N,2}(z)$ are considered as error terms.
    Complex powers $(P-i\varepsilon)^{-\alpha}$ are then expressed as a contour formula involving the Feynman resolvent, namely for $\varepsilon >
 0$,
\begin{equation} \label{eqcontourintro} (P -i\epsilon)^{-\alpha} = \frac{1}{2 \pi i } \int_{\gamma_{\epsilon}} (z -i\epsilon)^{-\alpha}(P - z)^{-1} dz,
\end{equation}
considered as an operator in $\pazocal{B}\left(L^2(M;E)\right)$ for $\Re (\alpha) > 0$. Note that it is indispensable that the estimates for the resolvent, parametrix and error terms be \textit{uniform in $z$} with sufficient decay along the infinite complex contour $\gamma_{\varepsilon}$. Our aim is to precisely make sense of the complex powers as distributions near the diagonal that admit a well-defined restriction to the diagonal in $M \times M$.

The main difficulty of the higher-signature pseudo-Riemannian case is that the Lorentzian causal argument used to control the errors is no longer available. In the Lorentzian setting, one can work in causally convex neighborhoods, so that singularities created by the cut-off in the Hadamard parametrix cannot leave the local region and return to the diagonal. In mixed signature, however, the null cone does not in general determine a distinguished future/past decomposition, and hence there is no canonical causal order or corresponding notion of causal convexity with the same propagation property. We therefore replace this argument by a \emph{geometric no-return} property: using the non-trapping structure of the Hamilton flow associated to the principal symbol of $P$, we show that the return relation, denoted by $\pazocal{R}'_{\chi}$, does not meet the phase-space diagonal. Namely, by Lemma \ref{lemconv}, it reads $
    \pazocal{R}'_{\chi} \cap D'_+
    =
    \varnothing$, while the wavefront set of the first error term is shown to lie within the return relation $\pazocal{R}'_{\chi}$ in the proof of Theorem \ref{thmRz}.
    
This no-return property is precisely what is needed, through Hörmander’s pullback criterion, to define the diagonal restriction of the contour-integrated remainder appearing in the complex powers (\ref{eqcontourintro}) after replacing the resolvent by the Hadamard approximation identity (\ref{approxid}).

We also introduce a stronger condition in Definition \ref{defrelforconvdef}, namely \emph{relative forward convexity}, which is needed to recover the sharper form of the Lorentzian local expansion. Under this additional hypothesis, the propagated cut-off error does not return to a whole neighborhood of the base diagonal. Consequently, the remainder admits a well-defined diagonal pullback, and is as regular as we want it to be. In this sense, the no-return argument gives the minimal geometric control needed for the spectral zeta density, while relative forward convexity is needed to get the stronger regularity statement familiar from the Lorentzian framework.

In the end, the remaining step is the computation of the local residues. The study of the on-diagonal restriction of the complex powers $\Pa$ is reduced to the analysis of the elementary family $\mathbf F_\alpha^\kappa(x,y)$ at $x=y$.
This value is computed using the Schwinger parameter representation and the Fresnel formula for the quadratic form $Q_\kappa$. We prove it admits meromorphic continuation in $\alpha \in \C$ and we compute its residues for arbitrary $n$. Note that the parity of $n$ only comes into play when we compute the residues at the very end.

\subsection{Related work and context} We conclude the introduction with a few remarks placing the present work in the context of related developments.

The breakthrough that led to the exploration of the connection between local geometric invariants and global spectral objects in indefinite signature was Vasy’s proof of the essential self-adjointness of the wave operator on non-trapping Lorentzian scattering spaces \cite{Vasy2017EssentialSO}. This result builds on the framework of Lorentzian sc-metrics, first introduced in \cite{baskin2014asymptoticsradiationfieldsasymptotically}, and then studied in more detail in \cite{Hintz:2013lka, Gell_Redman_2016}, together with the analysis of the Feynman propagator. We note that, more recently, Baskin--Doll--Gell-Redman \cite{BaskinDollGellRedman2025} extended the construction of Feynman propagators to massive Klein--Gordon operators with asymptotically static perturbations, using the $3\mathrm{sc}$-calculus. Essential self-adjointness was extended by
Nakamura–Taira \cite{9nakamura2019essentialselfadjointnessrealprincipal} to other differential operators of real principal type on long-range perturbations of Minkowski space, and recently to broader classes of asymptotically Minkowski spacetimes by Jia--Molodyk--Sussman \cite{Jia2026}. The present work builds on the spectral results of \cite{dang2024diracoperatorslocalinvariants} for non-trapping operators that are not necessarily self-adjoint, but whose failure of self-adjointness is controlled in the sc-calculus; see \cite[\S~2.5]{dang2024diracoperatorslocalinvariants}.

Related developments that connect the global analytic theory with local geometric invariants include, alongside the work of Dang--Wrochna \cite{Dang:2020jch}, a Lorentzian Atiyah-Patodi-Singer index theory and its local counterpart due to Bär--Strohmaier \cite{article}. In a different direction, Strohmaier–Zelditch \cite{Strohmaier_2021} give a  spectral-theoretical procedure to recover the scalar curvature, and hence a spectral action, in the stationary case. We also mention the work of Zeitoun \cite{zeitoun2025feynmanspectralactionwave}, whose construction is closely related to that developed in the present paper and which defines a spectral action of the wave operator on asymptotically de Sitter spaces. A common feature of these developments is the central rôle played by Feynman inverses and by the Hadamard parametrix.

Finally, it is worth emphasizing the connection between these results and some developments in Lorentzian noncommutative geometry; see for instance \cite{Verch:2011ry, BESNARD2018292, Bizi}, and \cite{Filaci}. These works are primarily concerned with extending the algebraic framework of spectral triples to the Lorentzian setting. In particular, although these constructions rely on a passage from Euclidean signature and on additional symmetry or analyticity assumptions, we mention results on the spectral action due to Devastato--Farnsworth--Lizzi--Martinetti \cite{Devastato:2017rlo} and Martinetti--Singh \cite{Martinetti:2019wgh}. Closer to the setting considered here, recent developments by Nieuviarts \cite{nieuviarts2026emergencetimetwistedspectral} and Martinetti \cite{martinetti2026twistedstandardmodelkrein} show the emergence of pseudo-Riemannian signature from twisted spectral triples. In this context, formula~(\ref{psRspectral}) provides local geometric coefficients that are naturally relevant to a pseudo-Riemannian spectral action.

\subsection{Structure of the paper} In Section~\ref{geomsetup}, we introduce the framework of scattering geometry and recall the basic features of the scattering calculus, together with the notion of uniform operator wavefront set that will be used throughout. At the beginning of Section~\ref{pseudo-R}, we consider a scattering differential operator $P\in\mathrm{Diff}_{\mathrm{sc}}^{2,0}(M;E)$ of Laplace-type on a pseudo-Riemannian manifold $(M,g)$ of signature $(\kappa,n-\kappa)$. We recall the spectral and Fredholm properties of $P$ established in  \cite{dang2024diracoperatorslocalinvariants}, which will be used throughout the paper.

We then define the elementary family $F_\alpha^\kappa$ in arbitrary signature in Section \ref{sectelemfam}, and we show that it enjoys the same properties as in the Lorentzian case. Section \ref{transeq} defines the uniform Hadamard parametrix $H_N(z)$. In Section~\ref{sectlocalexppseudoRcxpow}, we relate the resolvent $\resP$ to the parametrix through an operator identity that involves the error terms. In Section~\ref{sectlocexpcxpow}, we define the complex powers as distributions and show that the restriction to the diagonal is justified by the no-return argument. The residues of the meromorphically continued diagonal family are computed in Section \ref{sectresidues}, for arbitrary $n$. The contribution of the scalar curvature in the residues constitutes Section \ref{sectsclarcurvcontrib}.

Finally, Section \ref{sectapp} is devoted to applications of the residue formula. First, we consider the wave operator in the odd-dimensional Lorentzian case; second, we compute the residues of the complex powers of the square of the scattering Dirac operator in arbitrary signature and dimension.

\section{scattering-geometry and analytic setup } \label{geomsetup}

Let $M$ be a smooth manifold of dimension $n \in \N$, and $g$ a pseudo-Riemannian metric of signature $(\kappa, n -\kappa)$, where $\kappa \in \{0,1, \ldots, n\}$. We denote by $L^2(M)$ the canonical $L^2$ space associated to the volume density $d\mathrm{vol}_g$ of $(M,g)$, namely its norm is
$$ \|u\| := \left( \int_M \vert u(x) \vert^2 \,d\mathrm{vol}_g \right)^{\frac{1}{2}}.$$
We assume that $M$ is the interior of a compact manifold with boundary $\, \overbar{M}$.

Let $\rho$ be a boundary defining function for $\, \overbar{M}$, i.e. $\rho$ is smooth, $\rho > 0$ on $M$, $\partial \, \overbar{M} = \{ \rho=0 \}$ and $d\rho$ does not vanish on $\partial \, \overbar{M}$. By the collar neighborhood theorem, there exists a neighborhood $W$ of the boundary $\partial \, \overbar{M}$ in $\, \overbar{M}$, $\epsilon > 0$ and a diffeomorphism $\phi : [0,\epsilon) \times \partial \, \overbar{M} \rightarrow W$ such that $\rho \circ\phi$ agrees with the projection to the first component of $[0,\epsilon) \times \partial \, \overbar{M}$. Thus, when we consider points close to the boundary we will use local coordinates of the form $(\rho, y_1, \dots, y_{n-1})$ where $y =  (y_1, \dots, y_{n-1})$ are local coordinates on $\partial \, \overbar{M}$ (we drop $\phi$ in the notation).

\begin{example} [Radial compactification] \label{radial} \textnormal{The \textit{radial compactification} of $\R^n$ is defined by the manifold with boundary $\overline{\R^n} = \R^n \sqcup \, \mathbb{S}^{n-1}$, described by the following construction. Let $x \in \R^n$ be non-zero, and let us consider it in polar coordinates $x = (r, \omega)$, where $r > 0 $ and $\omega \in \mathbb{S}^{n-1}$. Then 
$$ \overline{\R^n} := \left( \R^n \, \cup \, \left( [0, \infty)_{\rho} \times \mathbb{S}^{n-1} \right) \right) / \sim,$$
where $\R^n \ni x = (r, \omega) \sim (r^{-1}, \omega)$. If we put $\rho := r^{-1}$, it follows that $\rho^{-1}(0) \cong \mathbb{S}^{n-1}$, considered as the ``sphere at infinity'', namely the boundary of $\overline{\R^n}$. The manifold with boundary $\, \overline{\R^n}$ is diffeomorphic to the unit ball $\pazocal{B}^n$ of $\, \R^n$ (and in particular $\overline{\R}$ is a closed segment).}
\end{example}

\emph{Notation.} By construction, any scattering space is defined with respect to a manifold with boundary. Such a manifold will be $\, \overbar{M}$ in our case. To make the notation less cluttered, we will write $M$ instead of $\, \Mb$ when the scattering space is explicitly denoted with the subscript or superscript $\mathrm{sc}$ (which is always the case), such as $\Psi_{\mathrm{sc}}^{m,l} (M)$ or ${}^{\mathrm{sc}}S^*M$ for instance. The only exceptions to this rule concern the algebra of scattering vector fields $\pazocal{V}_{\mathrm{sc}}(\, \Mb)$, and the scattering tangent and cotangent bundles ${}^{\mathrm{sc}}T\, \Mb$ and ${}^{\mathrm{sc}}T^* \, \Mb$, still denoted with $\, \Mb$.

\subsection{Scattering geometry}

We start with a brief introduction to Melrose's \emph{scattering geometry}, first introduced in \cite{Melrose1994Spectral} for the study of the Laplacian on asymptotically Euclidean spaces. The algebra of \emph{scattering vector fields} \( \pazocal{V}_{\text{sc}}(\, \overbar{M}) \) is a Lie subalgebra of the algebra \( \pazocal{V}(\, \overbar{M}) \) of vector fields defined by
$$
\pazocal{V}_{\text{sc}}(\, \overbar{M}) = \rho\pazocal{V}_{\text{b}}(\, \overbar{M}),
$$
where $\rho$ is the boundary defining function of $\, \overbar{M}$, and $\pazocal{V}_{\text{b}}(\, \overbar{M})$ is the subalgebra of vector fields tangent to the boundary, namely 
$$
\pazocal{V}_{\text{b}}(\, \overbar{M}) = \{ V \in \pazocal{V}(\, \overbar{M}) \, ; \, V \rho\in \rho C^\infty(\, \overbar{M}) \},
$$
where $C^\infty(\, \overbar{M})$ denotes the space of smooth functions on $\, \overbar{M}$, meant in the usual sense of being smoothly extendible across the boundary $\partial \, \overbar{M}$. In local coordinates $(\rho, y_1, \ldots, y_{n-1})$, bases for these two algebras at a point \( p \) are respectively given by $$
\text{span}_\mathbb{R} \left\{ \rho\partial_\rho, \partial_{y_1}, \ldots, \partial_{y_{n-1}} \right\} \, \, \, \mathrm{and} \, \, \, \text{span}_\mathbb{R} \left\{ \rho^2 \partial_\rho, \rho\partial_{y_1}, \ldots, \rho\partial_{y_{n-1}} \right\}.$$
Therefore, over the ring of smooth functions $C^{\infty}(\, \overbar{M})$, the algebras $\pazocal{V}_{\text{b}}(\, \overbar{M})$ and $\pazocal{V}_{\text{sc}}(\, \overbar{M})$ can be considered as locally free sheaves of rank $n = \mathrm{dim}(M)$. By the smooth Serre--Swan theorem, they are the sections of well-defined vector bundles, namely the \textit{b} and \textit{scattering tangent bundles} ${}^{\text{b}} T\, \overbar{M}$ and \( {}^{\mathrm{sc}} T\, \overbar{M}\) respectively; see for instance \cite[Prop.~3.1--Thm 3.4]{VARILLY1993223}.

In local coordinates $(\rho, y_1, \ldots, y_{n-1})$, bases for ${}^{\text{b}} T_p \, \overbar{M}$ and \( {}^{\mathrm{sc}} T_p \, \overbar{M} \) at a point \( p \) are given by $$
{}^{\text{b}} T_p \, \overbar{M} = \text{span}_\mathbb{R} \left\{ \rho\partial_\rho, \partial_{y_1}, \ldots, \partial_{y_{n-1}} \right\}, \quad {}^{\mathrm{sc}} T_p \, \overbar{M} = \text{span}_\mathbb{R} \left\{ \rho^2 \partial_\rho, \rho\partial_{y_1}, \ldots, \rho\partial_{y_{n-1}} \right\},$$
and 
$$\pazocal{V}_{\text{b}}(\, \overbar{M}) = C^{\infty}(\,\overbar{M}; {}^{\text{b}} T \, \overbar{M}), \quad \pazocal{V}_{\text{sc}}(\, \overbar{M}) = C^{\infty}(\,\overbar{M}; {}^{\mathrm{sc}} T \, \overbar{M}).$$

The \emph{b} and \emph{scattering cotangent bundles} \( {}^{\text{b}} T^* \, \overbar{M} \) and \( {}^{\mathrm{sc}} T^* \, \overbar{M} \) are the dual bundles to \( {}^{\text{b}} T \, \overbar{M} \) and \( {}^{\mathrm{sc}} T \, \overbar{M} \), with bases in local coordinates given by
\[
{}^{\text{b}} T_p^* \, \overbar{M} = \text{span}_\mathbb{R} \left\{ \frac{d\rho}{\rho}, dy_1, \ldots, dy_{n-1} \right\}, \quad {}^{\mathrm{sc}} T_p^* \, \overbar{M} = \text{span}_\mathbb{R} \left\{ \frac{d\rho}{\rho^2}, \frac{dy_1}{\rho}, \ldots, \frac{dy_{n-1}}{\rho} \right\}.
\]
Note that the three vector bundles $T^*\, \overbar{M}$, ${}^{\text{b}} T^* \, \overbar{M}$ and ${}^{\mathrm{sc}} T^* \, \overbar{M}$ are isomorphic to $T^*M$ over the interior of $\, \Mb$.

\begin{definition} \label{scpseudospace}
    We say that $g$ is an \emph{sc-metric} if it is a non-degenerate smooth section of the fiberwise symmetrized tensor bundle ${}^{\mathrm{sc}} T^* \, \overbar{M} \otimes_s {}^{\mathrm{sc}} T^* \, \overbar{M}$. The manifold $(M,g)$ is a \emph{pseudo-Riemannian scattering space} of signature $(\kappa, n - \kappa)$ if $M$ compactifies into a manifold with boundary $\, \Mb$ such that $g$ is a pseudo-Riemannian metric over the interior $M$ that smoothly extends to an sc-metric. Here, the convention is that $\kappa$ denotes the number of ``pluses''.
\end{definition}
In particular, for $\kappa = 1$, we will say that $(M,g)$ is a \emph{Lorentzian scattering space}.

The volume density of $(M,g)$ then extends to an
sc-density on $\, \overbar{M}$, i.e.\ in local coordinates $(\rho,y)$ it is of
the form $$
    \mu(\rho,y)
    \left|
        \rho^{-2}\,d\rho\,\rho^{-n+1}\,dy
    \right| $$
for some $\mu \in C^\infty(\, \overbar{M})$. Note that, whenever we consider a pseudo-Riemannian scattering space $(M,g)$, we will implicitly fix an sc-density $d \mathrm{vol}_g$. 

For the sake of generality, in the following discussion on pseudodifferential operators one could actually take any smooth manifold $\, \overbar{M}$ with boundary and equip it with an sc-density to define the Hilbert space $L^2(\, \Mb)$.

\subsection{The sc-calculus}

 The fibers of ${}^{\mathrm{sc}} T^*\, \overbar{M}$ are radially compactified for microlocalization, which produces a manifold with corners, denoted by $\overline{{}^{\mathrm{sc}} T^* M}$. It has two boundary hypersurfaces, the first one being \textit{base infinity}, denoted by $\partial {}^{\mathrm{sc}} T^* \, \overbar{M}$, and the second one \textit{fiber infinity}, denoted by $\partial \overline{{}^{\mathrm{sc}} T^*} M$, so that 
$$\partial \overline{{}^{\mathrm{sc}} T^*M} = \partial {}^{\mathrm{sc}} T^* \, \overbar{M} \, \cup \partial \overline{{}^{\mathrm{sc}} T^*} M,$$
which includes the corner 
$$ \partial {}^{\mathrm{sc}} T^* \, \overbar{M} \, \cap \partial \overline{{}^{\mathrm{sc}} T^*} M \neq \emptyset.$$ We will sometimes refer to fiber infinity as the \textit{sc-cosphere bundle} of $M$, denoted by ${}^{\mathrm{sc}}S^*M$, and denote the corner by ${}^{\mathrm{sc}}S_{\partial M}^*M$. As boundary defining function of fiber infinity ${}^{\mathrm{sc}}S^*M$ we take
\[
    \rho_\infty
    =
    \langle (\tau,\mu)\rangle^{-1}
    =
    \bigl(|(\tau,\mu)|^2+1\bigr)^{-1/2},
\]
where \( |\cdot| \) denotes the length with respect to an arbitrarily chosen
Riemannian sc-metric. The coordinates on the fibers of ${}^{\mathrm{sc}}T^*\, \Mb$ are denoted by $(\tau,\mu_1,\ldots,\mu_{n-1})$, sc-dual to
$(\rho,y_1,\ldots,y_{n-1})$, i.e.\ sc-covectors are written as
\begin{equation*}
    \tau \frac{d\rho}{\rho^2}
    +
    \sum_{j=1}^{n-1}
    \mu_j \frac{dy_j}{\rho}
    :=
    \tau \frac{d\rho}{\rho^2}
    +
    \mu\cdot\frac{dy}{\rho},
\end{equation*}
while  $\rho$ could be considered as a boundary defining function of base infinity. For $s,\ell\in\mathbb{R}$, the class of \emph{scattering symbols} of order
$(s,\ell)$ is denoted by $
    S_{\mathrm{sc}}^{s,\ell}(T^*M)$ (strictly speaking, it must be denoted by $S^{s,\ell}({}^{\mathrm{sc}} T^* \, \overbar{M}))$.
Away from $\partial \,\overbar{M}$, this agrees with the usual symbol class
$S^s(T^*M)$. Near the boundary, a symbol
$a\in S_{\mathrm{sc}}^{s,\ell}(T^*M)$ satisfies
\begin{equation} \label{symbest}
    \left|
        (\rho\partial_\rho)^j
        \partial_y^\alpha
        \partial_\tau^k
        \partial_\mu^\beta
        a(\rho,y,\tau,\mu)
    \right|
    \leq
    C_{jk\alpha\beta}\,
    \rho^{-\ell}
     \langle (\tau,\mu)\rangle^{s-k-|\beta|},
    \qquad
    j,k\in\mathbb{N},\quad
    \alpha,\beta\in\mathbb{N}^{n-1}.
\end{equation}

The basic model is $\mathbb{R}^n$, with standard coordinates
$(x,\xi)$ on $T^*\mathbb{R}^n$. Writing
\[
    x=(r,y),\qquad
    \rho=r^{-1},\qquad
    \xi=(\tau,\rho^{-1}\mu),
\]
we set
\[
    \langle x\rangle=(1+|x|^2)^{1/2},
    \qquad
    \langle\xi\rangle=(1+|\xi|^2)^{1/2}.
\]
One sees that the estimates (\ref{symbest}) are equivalent to 
\[
    \left|
        \partial_x^\alpha\partial_\xi^\beta a(x,\xi)
    \right|
    \leq
    C_{\alpha\beta}
    \langle x\rangle^{\ell-|\alpha|}
    \langle\xi\rangle^{s-|\beta|},
    \qquad
    \alpha,\beta\in\mathbb{N}^n.
\]
 The class of \emph{scattering pseudodifferential operators}, denoted by $\Psi_{\mathrm{sc}}^{s,\ell}(M)$, is obtained from symbols $
    a\in S_{\mathrm{sc}}^{s,\ell}(T^*M)$
by reducing locally to the standard quantization of symbols on
$T^*\mathbb{R}^n$. This construction requires a partition of unity
$\{\psi_i\}_i$ subordinate to a finite coordinate covering of $M$, together
with suitable coordinate identifications near $\partial \,\overbar{M}$. Formally, it reads 
$$\Psi_{\mathrm{sc}}^{s,\ell}(M) := \mathrm{Op}\left( S_{\mathrm{sc}}^{s,\ell}(T^*M)\right) + \Psi_{\mathrm{sc}}^{-\infty,-\infty}(M),$$
where $\mathrm{Op}$ is a quantization map. The class of regularizing operators $\Psi_{\mathrm{sc}}^{-\infty,-\infty}(M)$ consists of pseudodifferential operators whose Schwartz kernels $K_A(x,x')$ are smooth and rapidly decreasing, together with all their derivatives, as either variable tends to infinity.
Different choices lead to the same operator modulo $
    \Psi_{\mathrm{sc}}^{s-1,\ell-1}(M).$

\bigskip

\smallskip

In the scattering calculus, the principal symbol of
$A\in\Psi_{\mathrm{sc}}^{s,\ell}(M)$ is the equivalence class of its symbol
in the quotient
\[
    S_{\mathrm{sc}}^{s,\ell}(T^*M)
    \big/
    S_{\mathrm{sc}}^{s-1,\ell-1}(T^*M).
\]
It is often useful to restrict attention to the smaller class of
\emph{classical scattering pseudodifferential operators}, obtained from
classical symbols. The subclass of classical symbols, characterized by having a joint one-step polyhomogeneous expansion in
powers of $\langle x\rangle$ and $\langle \xi\rangle$ at infinity, can be invariantly defined as $ \rho^{-\ell}\rho_\infty^{-s} C^\infty\bigl(\overline{{}^{\mathrm{sc}}T^*M}\bigr).$ 
In particular, if $A\in\Psi_{\mathrm{sc}}^{0,0}(M)$ is classical, its
principal symbol may be identified with the restriction of
$a$ to 
$    \partial\overline{{}^{\mathrm{sc}}T^*M}$.
For operators of arbitrary order, one similarly obtains a natural
identification of the principal symbol with a suitably rescaled function
on $\partial \overline{{}^{\mathrm{sc}}T^*M}$.

The \emph{microsupport} $\operatorname{WF}'_{\mathrm{sc}}(A)$ of
$A\in\Psi_{\mathrm{sc}}^{s,\ell}(M)$ is the complement of the set of
points $
    q\in\partial\overline{{}^{\mathrm{sc}}T^*M}$
for which, in a neighborhood of $q$, the full symbol of $A$ belongs to $
    S_{\mathrm{sc}}^{-N,-L}(T^*M)$ for every $N,L\in\mathbb{R}$. If $A$ is classical, its elliptic set is
denoted by $\operatorname{Ell}_{\mathrm{sc}}(A)$, and its characteristic
set by $\overline{\Sigma}_{\mathrm{sc}}(A)$, defined as the closure of the zero set of the principal symbol,   so that
\[
    \operatorname{Ell}_{\mathrm{sc}}(A)
    =
    \partial\overline{{}^{\mathrm{sc}}T^*M}
    \setminus
    \overline{\Sigma}_{\mathrm{sc}}(A).
\]

It is useful to allow the weight order
$\ell$ to vary on $\overline{{}^{\mathrm{sc}}T^*M}$. Besides the natural modifications in the definitions of
$S_{\mathrm{sc}}^{s,\ell}(T^*M)$ and
$\Psi_{\mathrm{sc}}^{s,\ell}(M)$, one slightly weakens the symbol estimates.
More precisely, one requires that, for some $\delta>0$,
\begin{equation*}
    \left|
        (\rho\partial_\rho)^j
        \partial_y^\alpha
        \partial_\tau^k
        \partial_\mu^\beta
        a(\rho,y,\tau,\mu)
    \right|
    \leq
    C_{jk\alpha\beta}\,
    \rho^{-\ell-\delta(j+k+|\alpha|+|\beta|)}
    \langle\xi\rangle^{s-k-|\beta|}.
\end{equation*}
This small loss avoids the logarithmic factors that would otherwise arise
when differentiating a variable order $\ell$. Apart from this point, the
modification has little practical effect, and we suppress it in the
notation.

\emph{Scattering Sobolev spaces.} Now let $s\geq0$ and $
    \ell\in
    C^\infty\!\left(
        \overline{{}^{\mathrm{sc}}T^*M}
    \right)$.
The \emph{weighted scattering Sobolev space of variable weight order} is
defined by
\[
    H_{\mathrm{sc}}^{s,\ell}(M)
    =
    \left\{
        u\in L^2(\, \Mb)
        \ \vert \
        Au\in L^2(\, \Mb)
    \right\},
\]
where $A\in\Psi_{\mathrm{sc}}^{s,\ell}(M)$ is any classical elliptic operator, that is,
\[
    \operatorname{Ell}_{\mathrm{sc}}(A)
    =
    \partial\overline{{}^{\mathrm{sc}}T^*M}.
\]
Different choices of $A$ give equivalent norms. In particular, one may choose $A$ invertible and set
\[
    \|u\|_{s,\ell}
    :=
    \|Au\|_{L^2}.
\]
For a positive integer $s$, this agrees with the usual definition of
weighted scattering Sobolev spaces. For $s\leq0$, one may define
$H_{\mathrm{sc}}^{s,\ell}(M)$ by duality from
$H_{\mathrm{sc}}^{-s,-\ell}(M)$.

Thus, for arbitrary $s,\ell\in\mathbb{R}$ and any elliptic
$A\in\Psi_{\mathrm{sc}}^{s,\ell}(M)$, one may equivalently write
\[
    H_{\mathrm{sc}}^{s,\ell}(M)
    =
    \Bigl\{
        u\in
        \bigcup_{s',\ell'}
        H_{\mathrm{sc}}^{s',\ell'}(M)
        \ \vert \
        Au\in L^2(\, \Mb)
    \Bigr\}.
\]

\emph{Hamiltonian mechanics}. If $
    A \in \Psi_{\mathrm{sc}}^{m,\ell}(M),
    B \in \Psi_{\mathrm{sc}}^{m',\ell'}(M),$
and $a$, respectively $b$, denotes the principal symbol of $A$, respectively $B$, then the principal symbol of $i[A,B]$ is the Poisson
bracket $\{a,b\}$ associated with the symplectic structure on $T^*M$. 

In the case of $\mathbb{R}^n$, this is given by the familiar
formula
\[
    \{a,b\}
    =
    \sum_{j=1}^n
    \left(
        (\partial_{x_j}a)(\partial_{\xi_j}b)
        -
        (\partial_{\xi_j}a)(\partial_{x_j}b)
    \right)
    =: H_a b,
\]
which defines the Hamilton vector field $H_a$. In the fiber coordinates $
    (\tau,\mu_1,\ldots,\mu_{n-1})$
introduced above, a straightforward change of variables gives
\[
    H_a
    =
    \rho
    \left(
        (\partial_\tau a)
        (\rho\partial_\rho+\mu\cdot\partial_\mu)
        -
        (\rho\partial_\rho+\mu\cdot\partial_\mu)a\,\partial_\tau
        +
        (\partial_\mu a)\cdot\partial_y
        -
        (\partial_y a)\cdot\partial_\mu
    \right).
\]
In view of the homogeneity with respect to $\rho$ and $\rho_\infty$, it is
natural to introduce the rescaled Hamilton vector field
\[
    \overbar{H}_a
    :=
    \rho^{\ell-1}\rho_\infty^{m-1}H_a.
\]
If $a$ is classical, then $\, \overbar{H}_a$ extends smoothly to the
compactified scattering cotangent bundle as a b-vector field,
\[
    \overbar{H}_a
    \in
    \pazocal{V}_{\mathrm b}
    \bigl(\overline{{}^{\mathrm{sc}}T^*M}\bigr),
\]
and therefore generates a flow on
$\overline{{}^{\mathrm{sc}}T^*M}$, called the \textit{Hamilton flow}. For a general symbol $a \in S^{m,\ell}\bigl({}^{\mathrm{sc}}T^*M\bigr)$,
the vector field $\, \overbar{H}_a$ has only conormal coefficients of order
$(0,0)$, while remaining tangent to the boundary
    $\partial\,\overline{{}^{\mathrm{sc}}T^*M}$. \textit{Radial points} are the ones in $\partial\,\overline{{}^{\mathrm{sc}}T^*M}$ at which $\, \overbar{H}_a$ vanishes as a vector in $T\overline{{}^{\mathrm{sc}}T^*M}$.

\subsection{Operators on Hermitian vector bundles}

Let \(E,F\to \, \overbar M\) be complex vector bundles. The spaces $H_{\mathrm{loc}}^s(M;E)$ and $H_{\mathrm{c}}^s(M;E)$  denote the standard spaces on the \emph{boundaryless} manifold $M$, namely the sections of $E$ with Sobolev regularity $s$ and the compactly supported Sobolev sections of order $s$ respectively. The space $\pazocal{D}'(M;E)$ denotes the distributional sections of $E$ and $C_{\mathrm{c}}^{\infty}(M;E)$ is the space of smooth sections with compact support on $M$. 

Let $
    \pi_L,\pi_R:M\times M\longrightarrow M$
be the left and right projections. We denote $
    F\boxtimes E^*
    :=
    \pi_L^*F\otimes \pi_R^*E^*$, where $$
    (F\boxtimes E^*)_{(x,y)}
    =
    F_x\otimes E_y^*
    \simeq \operatorname{Hom}(E_y,F_x).$$
Consequently, the Schwartz kernel of an operator $
    A:C_c^\infty(M;E)\longrightarrow \pazocal D'(M;F)$
is a distributional section $
    K_A\in \pazocal D'(M\times M;F\boxtimes E^*)$.
We denote by $$
    \iota_\Delta:M\to M\times M,
    \qquad
    \iota_\Delta(x)=(x,x)$$ the diagonal map, where $\Delta$ is the diagonal in $M \times M$. Whenever the diagonal pullback is defined, one has
\begin{equation} \label{relpullbackdiag}
    \iota_\Delta^*K_A
    \in
    \pazocal D'(M;F\otimes E^*)
    \simeq
    \pazocal D'(M;\operatorname{Hom}(E,F)).
\end{equation}

Now, let $
\pi:{}^{\mathrm{sc}}T^* \, \overbar M \rightarrow \, \overbar M $ be the projection. For \(m,\ell\in\mathbb R\), we define the space of scattering symbols $S^{m,\ell}\left({}^{\mathrm{sc}}T^* \, \overbar{M};
\pi^* \operatorname{Hom}(E,F)\right)$, that we will denote by
$$S_{\mathrm{sc}}^{m,\ell}\left(T^*M;
\operatorname{Hom}(E,F)
\right),$$
as the space of symbols which verify in local scattering coordinates
\((\rho,y;\tau,\mu)\) and local trivializations of \(E\) and \(F\) the estimates (\ref{symbest}), i.e.
$$
\left|
(\rho\partial_\rho)^j
\partial_y^\alpha
\partial_\tau^k
\partial_\mu^\beta a
\right|
\leq
C_{j\alpha k\beta}\,
\rho^{-\ell}
\langle(\tau,\mu)\rangle^{m-k-|\beta|}.$$
Thus, for every \((x,\xi)\in{}^{\mathrm{sc}}T^* \, \overbar M\),
\[
a(x,\xi)\in\operatorname{Hom}(E_x,F_x).
\]
The classical symbol class is the subspace
\[
S_{\mathrm{cl}}^{m,\ell}\left( T^* M;
\operatorname{Hom}(E,F)
\right)
:=
\rho^{-\ell}\rho_\infty^{-m}
C^\infty\left(
\overline{{}^{\mathrm{sc}}T^* M};
\pi^*\operatorname{Hom}(E,F)
\right).
\]

On a trivializing neighborhood \(U\subset \, \overbar M\), a bundle-valued
symbol is represented by a matrix
\[
a=(a_{\alpha\beta}),
\qquad
a_{\alpha\beta}\in
S_{\mathrm{sc}}^{m,\ell}(T^*U),
\]
and is quantized entrywise, i.e. $
\operatorname{Op}(a)
=
\bigl(\operatorname{Op}(a_{\alpha\beta})\bigr)$.
Patching these local quantizations by means of bundle trivializations
and a partition of unity defines 
\[
\Psi_{\mathrm{sc}}^{m,\ell}(M;E,F).
\]
For $E = F$, this space will be denoted by $\Psi_{\mathrm{sc}}^{m,\ell}(M;E)$. 

We say that $A \in \Psi_{\mathrm{sc}}^{m,\ell}(M;E)$ is elliptic at $q = (x;\xi) \in\partial\overline{{}^{\mathrm{sc}}T^*M}$ if its principal symbol at $q$,
$$\sigma_m(A)(q) : E_x \rightarrow E_x,$$
is an isomorphism. All such elliptic points $q \in\partial\overline{{}^{\mathrm{sc}}T^*M}$ define the elliptic set $\mathrm{Ell}_{\mathrm{sc}}(A)$. The characteristic set $\overline{\Sigma}_{\mathrm{sc}}(A)$ is the complement $\overline{\Sigma}_{\mathrm{sc}}(A)
    =
    \partial\overline{{}^{\mathrm{sc}}T^*M}
    \setminus
    \operatorname{Ell}_{\mathrm{sc}}(A)$.

We now define \emph{scattering connections} on $E \to \, \overbar{M}$ as covariant derivatives 
$$\nabla : C^{\infty}(\, \Mb;E) \rightarrow C^{\infty}(\, \Mb;{}^{\mathrm{sc}}T^*\, \Mb \otimes E)$$ which satisfy the Leibniz condition $\nabla(f\cdot s) = df \otimes s + f \cdot \nabla s$, where $f \in C^{\infty}(\, \Mb)$ and $s \in C^{\infty}(\, \Mb;E)$. Note that $df$ refers to the image of the one-form in $C^{\infty}(\, \Mb;{}^{\mathrm{sc}}T^*\, \Mb)$ via the natural bundle map $T^*\, \Mb \to {}^{\mathrm{sc}}T^*\, \Mb$, induced by the inclusion $\pazocal{V}_{\mathrm{sc}}(\, \Mb) \subset \pazocal{V}(\, \Mb)$.

The algebra of scattering differential operators $\mathrm{Diff}_{\mathrm{sc}}^*\left(M,E\right)$ is the universal enveloping algebra of $\pazocal{V}_{\mathrm{sc}}(\, \overbar{M}) \otimes C^\infty \left(\, \overbar{M};
\operatorname{End}(E) \right) $ over $C^\infty(\, \overbar{M})$. That is, for \(m\in\mathbb N_0\),
$$\operatorname{Diff}_{\mathrm{sc}}^m
(M;E) = \Bigl\{\sum_{0 \leq \ell \leq m } c_{\ell} \nabla_{V_1} \cdots \nabla_{V_{\ell}}  \ \vert \ V_i \in \pazocal{V}_{\mathrm{sc}}(\, \Mb), c_{\ell} \in C^\infty \left(\, \overbar{M};
\operatorname{End}(E) \right) \Bigr\}.$$
For \(\ell\in\mathbb R\), we set
\[
\operatorname{Diff}_{\mathrm{sc}}^{m,\ell}
(M;E)
:=
\rho^{-\ell}
\operatorname{Diff}_{\mathrm{sc}}^m
(M;E) \subset \Psi_{\mathrm{sc}}^{m,\ell}(M;E).
\]
Assume now that the vector bundle $E$ is equipped with a positive-definite Hermitian metric. Together with the chosen positive sc-density, this defines
the Hilbert space \(L^2(\, \Mb;E)\).
Let $s\geq0$, $
    \ell\in
    C^\infty\!\left(
        \overline{{}^{\mathrm{sc}}T^*M}
    \right)$ and choose an elliptic invertible operator $
    A \in \Psi_{\mathrm{sc}}^{s,\ell}(M;E)$, with which we define
\[
    H_{\mathrm{sc}}^{s,\ell}(M;E)
    :=
    \left\{
        u\in L^2(\, \Mb;E)
         \ \vert \ 
        Au\in L^2(\, \Mb;E)
    \right\}.
\]
Equivalently, after choosing any scattering connection on \(E\), the same
spaces may be described for integer \(s\ge0\) by requiring all covariant
derivatives along sc-vector fields up to order \(s\) to lie in \(L^2\), with
the corresponding weight. Different choices of scattering connection give
equivalent norms, since two such connections differ by a smooth
\(\operatorname{End}(E)\)-valued scattering one-form.

The rest of the above scalar discussion follows in a similar way in the bundle setting. Note that the Hamilton vector field $H_a$ is well-defined as $H_a := H_{p_a}$ for an operator $A \in \Psi_{\mathrm{sc}}^{m,\ell}(M;E)$ that is principally scalar, i.e. its principal symbol $a$ is given by $$\sigma_m(A)(x;\xi) = p_a(x;\xi) \mathrm{Id}_{E_x}.$$

\subsection{The uniform operator wavefront set} \label{sectuniwf}
We recall the notion of uniform operator wavefront set used to characterize the wavefront set of the resolvent $(P-z)^{-1}$ in \cite[\S~3.3]{dang2024diracoperatorslocalinvariants}. We again assume that $E$ is a Hermitian bundle over $\, \Mb$, here considered over the interior $M$.

Let $Z \subset \mathbb{C}$ and let $h$ be a strictly positive function on $Z$. We will mostly consider $h$ of the form $h(z) = \langle z \rangle^{-p}$ where $p \in \N$. Suppose $G_z : C_{\mathrm{c}}^\infty(M;E) \to \Dd'(M;E)$ for $z \in Z$.

\begin{definition}
For $s \in \mathbb{R}$, we write
\[
G_z = \pazocal{O}_{H^{*} \to H^{*+s}}(h(z))
\]
if for all $\ell \in \mathbb{R}$, $h(z)^{-1} G_z$ is a uniformly bounded
family of continuous operators
$H_{\mathrm{c}}^{\ell}(M;E) \to H^{s+\ell}_{\mathrm{loc}}(M;E)$.
We write
$G_z = \pazocal{O}_{C^\infty \to C^\infty}(h(z))$
if
$G_z = \pazocal{O}_{H^{*} \to H^{*+s}}(h(z))$
for some $s \in \mathbb{R}$.
\end{definition}

The statement $G_z = \pazocal{O}_{C^\infty \to C^\infty}(\langle z \rangle^{-\infty})$ used throughout the paper will mean that for every $L\geq 0$, the family $G_z$ satisfies $G_z = \pazocal{O}_{C^\infty \to C^\infty}(\langle z \rangle^{- L})$.

Over the interior $M$ of $\, \overbar{M}$,
$\partial \overline{{}^{\mathrm{sc}}T^*}M$ is isomorphic to
$\partial \,\overline{T}^*M$, the boundary of the fiber compactification
of $T^*M$, that we will denote by $S^*M$. Let us denote by $\Psi^m(M;E)$ the class of properly supported
pseudodifferential operators of order $\leq m$ on $M$ in the sense of
the usual calculus on the boundaryless manifold $M$.
We say that $A \in \Psi^m(M;E)$ is elliptic at
$q = (x;\xi) \in S^*M$ if its principal symbol is invertible as an element of
\(\operatorname{End}(E_x)\).

\begin{definition} \label{defwf}
If
$G_z = \pazocal{O}_{C^\infty \to C^\infty}(h(z))$,
then its \emph{uniform operator wavefront set of order $s \in \mathbb{R}$}
is the set
\[
\operatorname{WF}'^{(s)}_{h(z)}(G_z)
\subset
S^*M \times S^*M
\]
defined as follows:
$(q_1,q_2) \notin \operatorname{WF}'^{(s)}_{h(z)}(G_z)$
if there exist $C_i \in \Psi^0(M;E)$, elliptic at $q_i$ $(i=1,2)$,
and such that
\[
C_1 G_z C_2^*
=
\pazocal{O}_{H^{*} \to H^{*+s}}(h(z)).
\]
\end{definition}

We will denote by $\operatorname{WF}'^{(s)}(G_z)$ the wavefront $\operatorname{WF}'^{(s)}_{h(z)}(G_z)$ where $h(z) = 1$ for all $z \in Z$.

\begin{lemma} \label{lemcompo}
Let
$G_{1,z} = \pazocal{O}_{C^\infty \to C^\infty}(h_1(z))$
and
$G_{2,z} = \pazocal{O}_{C^\infty \to C^\infty}(h_2(z))$,
and suppose that the operators $G_{1,z},G_{2,z}$ are properly supported
for all $z \in Z$. Then the composition
$$
G_{1,z} \circ G_{2,z}
=
\pazocal{O}_{C^\infty \to C^\infty}(h_1h_2(z))$$
is well-defined and satisfies
\[
\operatorname{WF}'^{(s)}_{h_1h_2(z)}
    (G_{1,z}G_{2,z})
\subset
\operatorname{WF}'^{(s)}_{h_1(z)}(G_{1,z})
\circ
\operatorname{WF}'^{(s)}_{h_2(z)}(G_{2,z}),
\]
where the composition of
$\Gamma_1,\Gamma_2
\subset
S^*M \times S^*M$
is defined by
\[
\Gamma_1 \circ \Gamma_2
=
\left\{
(q_1,q_2)
\in
S^*M \times S^*M
\;\middle|\;
\exists\, q_0 \in S^*M,\,
(q_1,q_0) \in \Gamma_1,\,
(q_0,q_2) \in \Gamma_2
\right\}.
\]
\end{lemma}
The proof is given by \cite[Lem.~3.7]{Dang:2020jch}. Recall that for $A \in \Psi^s(M;E)$, there is a notion of wavefront set $\wf'(A)$ that characterizes the codirections $q \in S^*M$ in which, microlocally, $A$ does not behave like a regularizing operator.

\begin{definition}
Let $\Lambda \subset T^*M \setminus o$ be conic.
Let $\{u_z\}_{z\in Z}$ be a family of distributions on $M$.
We write
\[
u_z = \pazocal{O}_{\Dd'_\Lambda}(h(z))
\]
if for all $A \in \Psi^0(M;E)$ satisfying
$\operatorname{WF}'(A) \cap \Lambda = \emptyset$
we have
\[
A u_z = \pazocal{O}_{C^\infty}(h(z)).
\]
\end{definition}

Finally, we define the \emph{primed relation} for subsets of $T^*(M\times M)\setminus o$.

\begin{definition} \label{defprimerel}
Let $\nu : T^*M \setminus o \to S^*M$ be the quotient map for the $\mathbb{R}_+$ action by fiberwise dilations. For each conic set $\Lambda \subset T^*(M \times M) \setminus o$, we define
\[
\Lambda' = \{(\nu(x_1;\xi_1), \nu(x_2; - \xi_2)) \mid (x_1,x_2; \xi_1, \xi_2) \in \Lambda, \ \xi_1 \neq 0, \ \xi_2 \neq 0 \},
\]
which is a subset of $S^*M \times S^*M$.
\end{definition}

\begin{lemma} \label{lemusforrN}
Suppose $\Lambda \subset T^*(M \times M) \setminus o$ is conic and $G_z = \pazocal{O}_{C^\infty \to C^\infty}(h(z))$. If the associated family of Schwartz kernels satisfies $G_z(\cdot) = \pazocal{O}_{\pazocal{D}'_\Lambda}(h(z))$, then $ \operatorname{WF}'^{(s)}_{h(z)}(G_z) \subset \Lambda'$
for all $s \in \mathbb{R}$.
\end{lemma}
The proof is given by \cite[Lem.~3.10]{Dang:2020jch}. We close this section with the definition of the notation $\pazocal{O}_{C^s}$, used throughout the paper, where the space $C^s\left(U;E \boxtimes E^*\right)$ denotes the sections of the bundle $E \boxtimes E^*$ of $C^s$-regularity.

\begin{definition}
    Let $U \subset M \times M$ open, and let $s > 0$. We will write that $G_z = \pazocal{O}_{C^s(U)}(h(z))$
    if $G_z = \pazocal{O}_{C^\infty \to C^\infty}(h(z))$ and if the associated family of Schwartz kernels $h(z)^{-1}G_z$ is uniformly bounded in $C^s\left(U;E \boxtimes E^*\right)$.
    \end{definition}

\smallskip

\section{complex powers on pseudo-riemannian scattering spaces} \label{pseudo-R}

Let $(M,g)$ be an $n$-dimensional pseudo-Riemannian scattering space of signature $(\kappa, n - \kappa)$, where $\kappa \in \{0,1, \dots, n \}$, in the sense of Definition \ref{scpseudospace}. 

Consider a smooth complex vector bundle $\pi : E \rightarrow \, \Mb$. The scalar case corresponds to $E$ being trivial. We assume that it is endowed with a fiberwise positive-definite Hermitian metric $h$, which after integration on $M$ with the volume density $\mathrm{vol}_g$ defines a scalar product $\langle \cdot, \cdot \rangle_E$.
The corresponding Hilbert space will be denoted by $L^2(\, \Mb;E)$. Consider now a classical scattering differential operator 
$$P \in \mathrm{Diff}_{\mathrm{sc}}^{2,0} (M;E) \subset \Psi_{\mathrm{sc}}^{2,0} (M;E)$$
which is of \emph{Laplace-type}, in other words with real scalar principal symbol $p$ given by 
$$ \sigma_2(P)(x;\xi)= - g_x^{-1}(\xi,\xi) \mathrm{Id}_{E_x}.$$
For a scalar potential in the case where $E$ is trivial, one may keep in mind the main example $P = \Delta_g + V$, where $V \in C^{\infty}(\, \Mb)$ and where $\Delta_g$ denotes the Laplace--Beltrami operator in signature $(\kappa, n - \kappa)$.

Now, consider the family of operators
$P-\lambda \in \Psi_{\mathrm{sc}}^{2,0}(M;E)$, where $\lambda \in \mathbb{C}$ is in a closed (not necessarily bounded) subset
\[
    Z_\varepsilon^+
    \subset
    \{ z \in \mathbb{C} \mid \operatorname{Im}(z) \geq \varepsilon \}
\]
of the upper half-plane, separated from the real line by at least $\varepsilon > 0$. 

Recall that, by the Laplace-type assumption, the principal symbol of $P$ is real.
The principal symbol of $P-\lambda$ at base infinity $\partial {}^{\mathrm{sc}} T^* \, \overbar{M}$ is therefore $p-\lambda$, with $p$ real
and $\mathrm{Im}(\lambda) \geq \varepsilon > 0$, so $P-\lambda$ is elliptic at base infinity (away from the corner). On the other
hand, at fiber infinity ${}^{\mathrm{sc}}S^*M$, $\lambda$ does not enter the principal symbol,
so the characteristic set is
\[
    \overline{\Sigma}_{\mathrm{sc}}(P-\lambda)
    =
    \overline{\Sigma}_{\mathrm{sc}}(P)
    \cap {}^{\mathrm{sc}}S^*M
    =:
    \overline{\Sigma}_0.
\]
Here we consider the rescaled Hamilton vector field associated with the principal symbol of $P$, denoted by $\, \overbar{H}_p$. The \emph{null bicharacteristics} are its integral curves within the characteristic set $\overline{\Sigma}_0$. 

We will write $q \sim q'$ if $q$ and $q'$ are linked
by a bicharacteristic in $\overline{\Sigma}_0$. 
The notation $q \succ q'$ means that $q \sim q'$ and $q$ comes after $q'$ along the flow. More precisely, if $t \mapsto \Phi_t$ denotes the flow of the rescaled Hamilton vector field $\, \overbar{H}_p$ associated to $p$ in $\overline{\Sigma}_0$, for $q,q' \in {}^{\mathrm{sc}}S^*M$ we write
$$q \sim q' \ \ (\textnormal{resp.} \ q \prec q', \ \textnormal{or} \ q \succ q')$$
if $q,q' \in \overline{\Sigma}_0$ and $\Phi_t(q) = q'$ for some $t \in \R$ (resp. $t > 0$ or $t<0$).

We now describe the radial source/sink structure at infinity. The sets
\(L_-\) and \(L_+\) are invariant submanifolds of the compactified
characteristic set which control the asymptotic behavior of null
bicharacteristics. The functions \(\rho_\pm\), \(\rho\), and \(\rho_\infty\)
measure distance to the radial set, to base infinity, and to fiber infinity,
respectively. The conditions below express that, in the appropriate Hamilton flow
direction, the rescaled Hamilton flow is expanding away from a source and
contracting towards a sink. 

\begin{definition} \label{defsourcesink}
We say that $L_- \subset \partial {}^{\mathrm{sc}} T^* \, \overbar{M}$,
resp.\ $L_+ \subset \partial {}^{\mathrm{sc}} T^* \, \overbar{M}$, are
\emph{sources}, resp.\ \emph{sinks} at infinity, if $L_\pm$ is a smooth
manifold transversal to ${}^{\mathrm{sc}}S^*_{\partial M}M$ and within
$\overline{\Sigma}_0$:

\begin{enumerate}[label=\alph*)]
    \item $dp \neq 0$ on $L_\pm$ and $\, \overbar{H}_p$ is tangent to $L_\pm$;

    \item there exists a quadratic defining function $\rho_\pm$ of $L_\pm$
    and a smooth function $\beta_{L_\pm}>0$ satisfying
    \[
        \mp \, \overbar{H}_p \rho_\pm
        =
        \beta_{L_\pm}\rho_\pm+s_\pm+r_\pm
    \]
    for some smooth functions $s_\pm,r_\pm$ such that $s_\pm\geq 0$ vanishes quadratically at $L_\pm$ and $r_\pm$ vanishes cubically at $L_\pm$;

    \item there exists
    $\beta_\pm \in C^\infty\bigl(\overline{{}^{\mathrm{sc}}T^*M}\bigr)$
    such that $\beta_\pm>0$ on $L_\pm$ and
    $\mp \, \overbar{H}_p\rho=\beta_\pm\rho$ modulo cubically vanishing terms;

    \item there exists
    $\beta_{\infty,\pm} \in C^\infty\bigl(\overline{{}^{\mathrm{sc}}T^*M}\bigr)$
    such that $\beta_{\infty,\pm}=0$ on $L_\pm$ and
    $\mp \, \overbar{H}_p\rho_\infty
    =\beta_{\infty,\pm}\beta_\pm\rho_\infty$
    modulo cubically vanishing terms.
\end{enumerate}
\end{definition}
By saying that
$\rho_\pm \in C^\infty\bigl(\overline{{}^{\mathrm{sc}}T^*M}\bigr)$
is a \emph{quadratic defining function} of $L_\pm$, we mean that
$\rho_\pm=\sum_i \rho_{\pm,i}^2$, for finitely many $\rho_{\pm,i}$ such that
$L_\pm=\bigcap_i\{\rho_{\pm,i}=0\}$ within
$\overline{\Sigma}_0\cap
\partial {}^{\mathrm{sc}} T^* \, \overbar{M}$,
and the differentials $d\rho_{\pm,i}$ are linearly independent on
$L_\pm\cap\overline{\Sigma}_0$. Note that any other boundary functions could be used here.

We can now formulate the \emph{non-trapping} condition imposed on $P$, which is the following.

\begin{definition} \label{Pnontrap}
    One says that the operator $P$ is \emph{non-trapping} if there exist sources $L_- \subset \partial {}^{\mathrm{sc}} T^* \, \overbar{M}$ and sinks $L_+ \subset \partial {}^{\mathrm{sc}} T^* \, \overbar{M}$ at infinity in the sense of Definition \ref{defsourcesink}, and, within $\overline{\Sigma}_0$, any bicharacteristic goes from $L_+$ to $L_-$, or from $L_-$ to $L_+$, or stays within $L_+$ or $L_-$. \\ We say that a Lorentzian scattering space is \emph{non-trapping} if the wave operator is non-trapping.
\end{definition}
We will specialize to the Lorentzian setting in Section \ref{sectLorentz} only. We also define the \emph{approximate self-adjointness condition} for the operator $P$, which will also allow us to use the results contained in \cite{dang2024diracoperatorslocalinvariants} by \cite[Thm.~3.9]{dang2024diracoperatorslocalinvariants}.

\begin{definition} \label{defsmallpertalt}
  We say that $P$ satisfies the \emph{approximate self-adjointness condition} if the following properties hold.
   \begin{enumerate}
    \item $P \in \Psi_{\mathrm{sc}}^{2,0} (M;E)$ is classical, and $P - P^* \in \Psi_{\mathrm{sc}}^{1,-1}(M;E)$,
    \item There exists $P_0$ such that $P-P_0$ is sufficiently small in $\Psi_{\mathrm{sc}}^{2,0} (M;E)$, $P_0$ is non-trapping in the sense of Definition \ref{Pnontrap} and $P_0-P^*_0 \in \Psi_{\mathrm{sc}}^{1,-1-\delta}(M;E)$ for some $\delta > 0$.
\end{enumerate} 
\end{definition}

Note that the adjoint operators are defined with respect to the inner product $\langle \cdot, \cdot \rangle_E$ defined above. We then gather all these assumptions in the following hypothesis.
\begin{hypothesis} \label{theH}
    $P \in \Psi_{\mathrm{sc}}^{2,0} (M;E)$ is classical and the characteristic flow of $p$ has sources and sinks $L_-$ and $L_+$ at scattering infinity and is non-trapping in the sense of Definition \ref{Pnontrap}. Moreover, if $P^*$ denotes the formal adjoint of $P$, either
\begin{equation}\label{thefirst}
    P - P^* \in \Psi^{1, -1-\delta}_{\mathrm{sc}}(M;E)
\end{equation}
for some $\delta > 0$, or $P$ satisfies the approximate self-adjointness condition of Definition \ref{defsmallpertalt}.
\end{hypothesis}
The first alternative (\ref{thefirst}) in Hypothesis \ref{theH} means that the failure of formal self-adjointness is a first-order term whose coefficients decay sufficiently fast at infinity. The extra $\delta>0$ makes that failure negligible in the leading radial estimates; see the \emph{positive commutator} proof in \cite[Prop~2.10--2.11]{dang2024diracoperatorslocalinvariants} and the discussion in \cite[\S~3.4]{dang2024diracoperatorslocalinvariants}.

In this setting, the resolvent of the operator $P$, denoted by $(P-z)^{-1}$, has been proven to have \textit{Feynman wavefront set}, denoted by $$\Lambda_{\mathrm{F}} \subset T^*(M \times M) \setminus o .$$
\begin{proposition} [Proposition 3.8 in \cite{dang2024diracoperatorslocalinvariants}] \label{resFeyn} Suppose $P \in \Psi_{\mathrm{sc}}^{2,0} (M;E)$ satisfies Hypothesis \ref{theH}. Then, for any $s \in \R$, the family of resolvents $\{ (P-z)^{-1} \}_{z \in \gamma_{\epsilon}}$ satisfies
    \begin{equation} \label{unirelation}
        \wf'^{(s)}_{\langle z\rangle^{-1}} \left( \resP \right) \subset \{ (q_1,q_2) \in S^*M \times S^*M \ \vert \ \, q_1 \succ q_2 \ \, \, \textnormal{or} \ \, q_1 = q_2 \} =: \Lambda'_{\mathrm{F}},
    \end{equation}
    using the notation defined in Definition \ref{defprimerel}.
\end{proposition}

 Recall the complex contour defined in \cite[Def. 5.5]{dang2024diracoperatorslocalinvariants}. Let $\tilde{\gamma}_{\epsilon}$ be the contour in the complex plane of the form 
\begin{equation*} 
    \tilde{\gamma}_{\epsilon} = e^{i(\pi - \theta)} \left]- \infty, \frac{\epsilon}{2} \right ] \cup \left\{ \frac{\epsilon}{2} e^{i\omega}, \,  \pi - \theta < \omega < \theta   \right\}   \cup e^{i\theta} \left[ \frac{\epsilon}{2}, \infty \right[
\end{equation*}
for some fixed $\theta \in \left] 0,  \frac{\pi}{2}\right [ $.

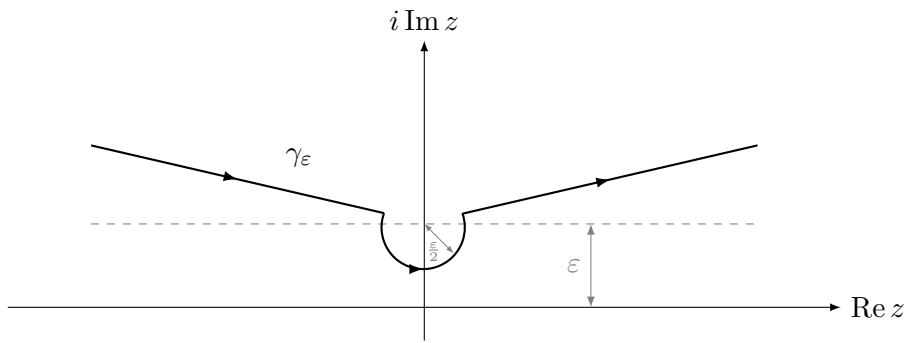
\begin{figure}[h!]
\centering
\begin{tikzpicture}[scale=2.2, >=latex]
\usetikzlibrary{decorations.markings}

\def\e{0.5} 
\def\r{0.25} 
\def\th{15}  
\def\d{0.01}

\tikzset{
  midarrow/.style={
    postaction={decorate,decoration={
      markings, mark=at position 0.5 with {\arrow{>}}
    }}
  }
}

\draw[->] (-2.5,0) -- (2.5,0) node[right] {$\mathrm{Re}\,z$};
\draw[->] (0,-0.2) -- (0,1.6) node[above] {$i\,\mathrm{Im}\,z$};

\draw[dashed, gray] (-2,\e) -- (2,\e);

\draw[<->, gray] (1,0) -- (1,\e);
\node[left, gray] at (1,\e/2) {$\varepsilon$};


\draw[thick,midarrow]
(-2, {\e + tan(\th)*(2 - \r*cos(\th))})
-- ({-\r*cos(\th)}, {\r*sin(\th) + \e});

\draw[thick,midarrow]
({-\r*cos(\th)}, {\r*sin(\th) + \e})
arc[start angle=-200, end angle=20, radius=\r];

\draw[thick,midarrow]
({\r*cos(\th) - \d}, {\r*sin(\th) + \e - \d*tan(\th)})
-- (2, {\e + tan(\th)*(2 - \r*cos(\th))});

\draw[<->, gray, line width=0.05pt]
(0,\e) -- ({\r*sin(45)}, {\e - \r*cos(45)})
node[midway, anchor=south east, xshift=4.5pt, yshift=-12pt, gray] {\tiny $\frac{\varepsilon}{2}$};


\node at (-0.75,{\e + 0.4}) {$\gamma_\varepsilon$};

\end{tikzpicture}

 \captionsetup{justification=centering}
\caption{The contour $\gamma_\varepsilon
= \widetilde{\gamma}_\varepsilon+i\varepsilon$}

\end{figure} 
 There are constants $c, C > 0$ such that, on the unbounded part of $\gamma_{\epsilon} $,
\begin{equation} \label{ineqinut}
  c \langle z \rangle  \leq \vert  \mathrm{Im}(z) \rvert \leq \langle z \rangle, \ \ \ z \in \gamma_{\epsilon}, \ \ \lvert z \rvert \geq  C,
\end{equation}
and the compact part of $\gamma_{\epsilon}$ lies in the resolvent set. Below, we present the main results of \cite{dang2024diracoperatorslocalinvariants}.

\smallskip

\begin{theorem} \label{propnonreal}
    Assume Hypothesis \ref{theH}. Then, the non-real spectrum of the minimal extension of $P$, i.e. the closure of \(P\) initially defined on $C^{\infty}_{\mathrm{c}}(M;E)$, consists at most of isolated points that lie in $\{ \vert \mathrm{Im(}\mathrm{ \lambda}) \vert \leq R \}$ for some $R > 0$. 
\end{theorem}

\begin{proof}
    This is \cite[Thm.~3.3]{dang2024diracoperatorslocalinvariants} for $P$ that is assumed non-trapping. Under the approximate self-adjointness condition, the result is given by \cite[Thm.~3.9]{dang2024diracoperatorslocalinvariants}. 
\end{proof}

For each $\epsilon > 0$ there is an upper-half-plane contour $\gamma_{\epsilon}$ of the form $\gamma_\varepsilon
= \widetilde{\gamma}_\varepsilon+i\varepsilon$, which is the one shown in \textit{Figure 1}, deformed by a compact amount if necessary, which avoids the non-real spectrum of $P$.

\begin{definition}[Feynman admissible scattering weight] \label{deffeynadmscwei}
Let \(\ell\in C^\infty(\overline{{}^{\mathrm{sc}}T^*M})\). We say that
\(\ell\) is a \emph{Feynman admissible scattering weight} if \(\ell\) is
non-increasing along the Hamilton flow on \(\overline{\Sigma}_0\) and satisfies
the radial threshold conditions
\[
    \ell|_{L_-}>-\frac12,
    \qquad
    \ell|_{L_+}<-\frac12 .
\]
\end{definition}

\begin{proposition}[Fredholm realization] \label{propfredreal}
Assume Hypothesis \ref{theH}. For \(z\in\mathbb C\) with \(\operatorname{Im}(z)>0\), define the \emph{Feynman spaces}
\[
    \pazocal{Y}_{\mathrm{sc}}^{s,\ell}
    :=
    H_{\mathrm{sc}}^{s,\ell}(M;E)
\]
and
\[
    \pazocal{X}_{\mathrm{sc}}^{s,\ell}(z)
    :=
    \Bigl\{u \in
        \bigcup_{s',\ell'}
H_{\rm sc}^{s',\ell'}(M;E)
        \ \vert \ 
        (P-z)u\in H_{\mathrm{sc}}^{s,\ell}(M;E)
    \Bigr\},
\]
where \(\ell\in C^\infty(\overline{{}^{\mathrm{sc}}T^*M})\) is a Feynman admissible scattering weight in the sense of Definition \ref{deffeynadmscwei}.
Then
\[
    P-z:
    \pazocal{X}_{\mathrm{sc}}^{s,\ell}
    \longrightarrow
    \pazocal{Y}_{\mathrm{sc}}^{s,\ell}
\]
is Fredholm, and, away from a discrete set of poles, invertible. Moreover, along $\gamma_{\epsilon}$, the resolvent $\resP$ is $\pazocal{O}(\langle z \rangle^{-1})$ as an operator in $\pazocal{B}(\pazocal{Y}_{\mathrm{sc}}^{s,\ell},\pazocal{X}_{\mathrm{sc}}^{s,\ell})$.
\end{proposition}

\begin{proof}
For $\mathrm{Im}(z)$ large, the Fredholm statement is the one appearing in \cite[Prop.~3.6]{dang2024diracoperatorslocalinvariants} and we wrote it here in the scattering framework (in \cite{dang2024diracoperatorslocalinvariants}, they use the \emph{resolved} setting, giving rise to the \emph{scattering quadratic calculus}); see \cite[\S~2.4]{dang2024diracoperatorslocalinvariants} for more details.

Note that the inequalities in $(\ref{ineqinut})$ hold on the unbounded part of $\gamma_{\epsilon}$. Its compact part can be deformed to avoid the finitely many poles it meets. On the unbounded part the global estimate, including the $\pazocal{O}(\langle z \rangle^{-1})$ decay, follows from \cite[Prop.~3.6]{dang2024diracoperatorslocalinvariants} as well; on the compact part it follows from holomorphy of the resolvent and compactness of the contour. Under the approximate self-adjointness condition, we apply the same reasoning after using \cite[Thm.~3.9] {dang2024diracoperatorslocalinvariants}.
\end{proof}

Let $\varepsilon > 0$. Within this framework, we define the complex powers of $P$ as
\begin{equation*} (P -i\epsilon)^{-\alpha} := \frac{1}{2 \pi i } \int_{\gamma_{\epsilon}} (z -i\epsilon)^{-\alpha}(P - z)^{-1} dz.
\end{equation*}
By Theorem \ref{propnonreal} and Proposition \ref{propfredreal}, together with the fact that we can deform the contour $\gamma_{\epsilon}$ by a compact amount that avoids the non-real spectrum of $P$, the above integral makes sense as an operator in $\pazocal{B}\left(L^2(M;E)\right)$ for $\Re (\alpha) > 0$. For the lower-half-plane family $(P+i\epsilon)^{-\alpha}$, one uses the reflected contour. 

Note that this definition depends on the choice of contour $\gamma_{\epsilon}$, but by \textnormal{\cite[Lem.~5.6]{dang2024diracoperatorslocalinvariants}} two different choices differ by a finite-rank operator with smooth Schwartz kernel. As a consequence, the local singular expansion of the complex powers as distributions and all residues obtained are independent of this choice.

\subsection{The elementary family in arbitrary signature} \label{sectelemfam}

Consider a flat pseudo-Euclidean space of metric signature $(\kappa,n-\kappa)$. Set
\begin{equation*} 
    Q_\kappa(\xi)
    =
    -\xi_1^2-\cdots-\xi_\kappa^2
    +\xi_{\kappa+1}^2+\cdots+\xi_n^2.
\end{equation*}

For $\operatorname{Im}(z)>0$ and initially for $\operatorname{Re}(\alpha)\gg0$, define
\begin{equation} \label{defflatF}
    F_\alpha^\kappa(z,x)
    :=
    \frac{\Gamma(\alpha+1)}{(2\pi)^n}
    \int_{\mathbb{R}^n}
    e^{i\langle x,\xi\rangle}
    \bigl(Q_\kappa(\xi)-z\bigr)^{-\alpha-1}
    \,d^n\xi,
\end{equation}
considered as a distributional inverse Fourier transform.
The branch of the logarithm is the standard branch on
$\mathbb{C} \setminus \left]-\infty,0\right]$. For $\mathrm{Im}(z) < 0$, the family of distributions is defined analogously. We will suppress the superscript $\kappa$ when the signature is fixed.

\begin{proposition} \label{propQk}
The analytic and microlocal properties of the family in
\textnormal{\cite[\S~4]{Dang:2020jch}} remain valid after replacing the Lorentzian quadratic
form by the pseudo-Riemannian $Q_\kappa$:
\begin{enumerate}
    \item
    $x \mapsto F_\alpha^\kappa(z,x)$ is a well-defined family of distributions on $\R^n$, holomorphic in $\alpha \in \C \setminus \{ -1, \ldots, - k, \ldots\}$
    for $\mathrm{Im}(z) > 0$.
    \item
    For the upper-half-plane boundary value, the uniform wavefront set
    is contained in
    \begin{equation} \label{flatwf}
        \Lambda_{0,+}
        =
        \left\{
            \bigl(x,\tau\,dQ_\kappa(x)\bigr)
            :
            Q_\kappa(x)=0,\ \tau<0
        \right\}
        \cup
        \bigl(T_0^*\mathbb{R}^n\setminus o\bigr).
    \end{equation}
    For the lower-half-plane family the inequality on $\tau$ is reversed.

    \item
    Away from $x=0$, the family is rapidly decreasing in $z$ in the normal
    topology associated with (\ref{flatwf}), i.e.
    \begin{equation} \label{normaltopFj}
        F^{\kappa}_{\alpha}(z) = \pazocal O_{\pazocal{D}'_{\Lambda_{0,+}}(\R^n \setminus \{ 0 \})}
    \bigl(\langle z\rangle^{-\infty}\bigr)
    \end{equation}
    for $z \in Z^+_{\epsilon}$, $\epsilon > 0$.
Moreover, given $s \in \R_+,p \in \N$, there is
    $N_0$ such that for all integers $N\geq N_0$,
    \begin{equation} \label{regFflat}
        F_N^\kappa(z)\in C_{\mathrm{loc}}^s(\mathbb{R}^n),
        \qquad
        F_N^\kappa(z)
        =
        \pazocal{O}_{C_{\mathrm{loc}}^s(\mathbb{R}^n)}\bigl(\langle z\rangle^{-p}\bigr)
    \end{equation}
    along the contour $\gamma_{\epsilon}$.

    \item
    The family is invariant under the full group
    $O(\kappa,n-\kappa)$ preserving $Q_\kappa$.
\end{enumerate}
\end{proposition}

\begin{proof}
The proof of \cite[\S~4.1-4.4]{Dang:2020jch} uses only the following properties
of the quadratic form: non-degeneracy, homogeneity, the fact that
$dQ\neq0$ on $Q^{-1}(0)\setminus\{0\}$, and the Fresnel formula for a
non-degenerate quadratic phase. All four properties hold for $Q_\kappa$.

\smallskip

\emph{Schwinger parameter and the Fresnel integral.} Let $\lambda=\alpha+1$ and $\operatorname{Im}(z)>0$. Since
$Q_\kappa(\xi)-z$ lies in the lower half-plane off the real axis, one may write initially for $\operatorname{Re}\lambda>0$
\begin{equation*} 
    \bigl(Q_\kappa(\xi)-z\bigr)^{-\lambda}
    =
    \frac{e^{i\pi\lambda/2}}{\Gamma(\lambda)}
    \int_0^\infty
    e^{-iu(Q_\kappa(\xi)-z)}
    u^{\lambda-1}\,du.
\end{equation*}

The factor $e^{iuz}$ provides exponential damping. The Fourier transform
of $e^{-itQ_\kappa}$ is a non-degenerate Fresnel Gaussian. Its absolute
homogeneity is $u^{-n/2}$, exactly as in Lorentzian signature, while its
constant phase depends on $\kappa$. Consequently,  the scalar parameter
integral has the same powers as in \cite[Lem.~4.4]{Dang:2020jch}. Therefore, all regularity thresholds depend on $n$ and
$\operatorname{Re}(\alpha)$, but not on the signature.

We now explain why the same wavefront estimate holds for \(Q_\kappa\). Let
\(\chi\in C_c^\infty(\mathbb R^n\setminus\{0\})\), and consider the Fourier
transform of \(\chi F_\alpha^\kappa(z,\cdot)\). Using the Schwinger parameter
representation and the Fresnel formula, one obtains, up to harmless constants,
an oscillatory integral with phase
\[
    \Phi(x,u;\xi)
    =
    \langle x,\xi\rangle+\frac{u}{4}Q_\kappa(x),
    \qquad u>0 .
\]
The critical equations
are
\[
    \partial_u\Phi=\frac14 Q_\kappa(x)=0,
    \qquad
    \partial_x\Phi=\xi+\frac{u}{4}dQ_\kappa(x)=0.
\]
Since \(u>0\), this is equivalent to
\[
    Q_\kappa(x)=0,
    \qquad
    \xi=\tau\,dQ_\kappa(x),
    \qquad
    \tau<0.
\]
Hence, for the upper-half-plane boundary value, the wavefront set is contained in 
\[
    \Lambda_{0,+}
    =
    \{(x,\tau dQ_\kappa(x)):Q_\kappa(x)=0,\ \tau<0\}
    \cup
    (T_0^*\mathbb R^n\setminus o).
\]
For the lower-half-plane family the sign of \(\tau\) is reversed.

The ultraviolet estimate is a standard non-stationary phase argument. If
\(\operatorname{supp}\chi\times V\) is disjoint from \(\Lambda_{0,+}\), then
the full differential
\[
    d_{x,u}\Phi
    =
    \left(
        \xi+\frac{u}{4}dQ_\kappa(x),
        \frac14 Q_\kappa(x)
    \right)
\]
does not vanish on the relevant support, for \(\xi\in V\). After a dyadic
decomposition in the parameter \(u\), one integrates by parts with the differential operator $\pazocal{L}$ (see \cite[Lem.~4.5, Step~1]{Dang:2020jch}). Repeated applications of $\pazocal{L}$ give arbitrary inverse powers of
\(\langle\xi\rangle\) on each dyadic piece, and the resulting dyadic series is
summable. This gives rapid decay of the Fourier transform of $
    \chi F_\alpha^\kappa(z,\cdot) $
in \(\xi\in V\), uniformly in the spectral parameter in the same sense as in the Lorentzian argument. This yields the rapid decay away from the critical relation (\ref{normaltopFj}).

For the infrared contribution we use the estimate in the proof of
\cite[Lem.~4.5, Step~2]{Dang:2020jch}: after the change \(u\mapsto u^{-1}\),
the \(x\)-dependent factor is smooth uniformly in $u$, while the
remaining \(u\)-integral is scalar and integrable. This yields rapid decay in
\(\xi\), uniformly in $z$, and, for \(\operatorname{Re}(\alpha)<n/2-1\), the stronger estimate uniform down to \(\operatorname{Im}(z)>0\), as in \cite[(4.48)]{Dang:2020jch}.

\emph{Bernstein--Sato continuation.}
Let $Q_\kappa(\partial_\xi)$ be the constant-coefficient operator obtained
from the inverse matrix of $Q_\kappa$. Differentiating
$\bigl(Q_\kappa(\xi)-z\bigr)^{-\lambda}$ twice and contracting gives a
relation of the form
\begin{equation} \label{diffQ}
    \bigl(
        Q_\kappa(\xi)Q_\kappa(\partial_\xi)+a(\lambda)
    \bigr)
    \bigl(Q_\kappa(\xi)-z\bigr)^{-\lambda}
    =
    b(\lambda)z
    \bigl(Q_\kappa(\xi)-z\bigr)^{-\lambda-1},
\end{equation}

where $a$ and $b$ are the same scalar polynomials as in
\cite[\S~4.4.1]{Dang:2020jch}. The signs in the Hessian cancel against those in
its inverse, so the polynomials depend only on $n$. After inverse Fourier
transform, (\ref{diffQ}) propagates the normal-topology estimates one
unit at a time in $\alpha$. This gives the same regularity exponents as
\cite[Thm.~4.8]{Dang:2020jch}, hence in particular (\ref{regFflat}).

Finally, invariance follows directly from (\ref{defflatF}) and
$| \det A \,|=1$ for $A\in O(\kappa,n-\kappa)$. Note that the poles of $F_{\alpha}^{\kappa}(z)$ are given by \cite[Cor.~4.2]{Dang:2020jch} that holds in arbitrary signature.
\end{proof}

\begin{remark}
\textnormal{The identities in normal coordinates used by Dang--Wrochna apply directly in the general pseudo-Riemannian setting; see \cite[\S~5.1.4]{Dang:2020jch}. The only modification needed for a general signature is the replacement of the flat Lorentzian form by the pseudo-Riemannian $Q_\kappa$ and the corresponding Fresnel integral.}
\end{remark}

\begin{remark}[Convex normal neighborhoods] \label{remconvexnei}
\textnormal{The existence of the neighborhood used below is independent of the signature. It is a theorem about the geodesic spray which says that every point of a smooth pseudo-Riemannian manifold has arbitrarily small reversible convex normal neighborhoods; see \cite{Minguzzi_2014}. Thus, the exponential map supplies
a smooth neighborhood of the zero section on which
\[
    (x,v)\mapsto\bigl(x,\exp_x v\bigr)
\]
is a diffeomorphism onto a neighborhood of the diagonal. This is all that
is required to solve the transport
equations along the distinguished short geodesic, and pull the flat
model back.}

\textnormal{There is some relevant literature for these two ingredients.
Kolk--Varadarajan treat Riesz distributions for quadratic forms of
general indefinite signature \cite{b9000f62-7943-3eb5-a554-ad1f98f1d129}, while Kamiński proves structural results for off-diagonal Seeley--DeWitt and related Hadamard coefficients on manifolds of arbitrary signature \cite{Kaminski:2019adk}. Extending
a single-valued signed squared geodesic distance to a substantially
larger neighborhood of the whole diagonal is a non-trivial patching
problem even in Riemannian and Lorentzian signature; see \cite{Moretti_2021}.
We avoid that issue by working in one such local diagonal neighborhood
and inserting the cut-off $\chi$.}
\end{remark}

\emph{Intrinsic pullback.} Recall the pullback of the elementary family $F_{\alpha}$ on $M\times M$.
Let $\pazocal{N}$ be a neighborhood of the zero section $o$ in $TM$
for which the map
\[
\pazocal{N} \ni (x;v)
\, \mapsto
\bigl(x,\exp_x(v)\bigr) \in M \times M
\]
is a local diffeomorphism onto its image, where
$\exp_x \colon T_xM \to M$ is the exponential geodesic map. Let 
\[
\U = \exp \pazocal{N} \subset M \times M,
\]
which is a neighborhood of $\Delta$, the diagonal in $M \times M$, so that the inverse map
\[
\U \ni (x_1,x_2)
\, \mapsto
\bigl(x_1;\exp_{x_1}^{-1}(x_2)\bigr) \in \pazocal{N}
\]
is a well-defined diffeomorphism.

Let $\Omega$ be an open subset of $M$. Let
$(e_0,\ldots,e_{n-1})$ be a pseudo-Riemannian orthonormal moving frame on
$\Omega$, and let $(s^\mu)_\mu$ be the corresponding
orthonormal moving coframe. We denote by $\epsilon_\mu$ the canonical basis of $\mathbb{R}^n$.
Such an orthonormal moving coframe $(s^\mu)_\mu$ defines, for $(x_1,x_2) \in \U$, the submersion
\begin{equation*}
\begin{aligned}
G \colon (x_1,x_2)  \mapsto
G^\mu(x_1,x_2)\epsilon_\mu =
\underbrace{s^\mu_{x_1}}_{\in T_{x_1}^*M}
(
\underbrace{\exp_{x_1}^{-1}(x_2)}_{\in T_{x_1}M}
)
\epsilon_\mu
\in \mathbb{R}^n .
\end{aligned}
\end{equation*}

Since the distribution $F_\alpha^\kappa(z,\cdot)$ is invariant under $O(\kappa,n-\kappa)$ by Proposition \ref{propQk}, its pullback by $G$ is independent of this choice. We obtain an intrinsic family
\begin{equation} \label{intrinF}
    \mathbf{F}_\alpha^\kappa(z,x) := G^*F_\alpha^\kappa(z,x) \in \pazocal{D}'(\U)
\end{equation}
whose uniform wavefront relation is contained in the Feynman wavefront set, by Proposition \ref{propQk} (2) and \cite[Lem.~5.2]{Dang:2020jch}, the latter argument being valid in arbitrary signature. This reads $$\wf'(\mathbf{F}_\alpha^\kappa(z)) \subset \Lambda'_{\mathrm{F}},$$ where the Feynman wavefront set $\Lambda_{\mathrm{F}}$ is defined in (\ref{unirelation}). 

In the bundle case, parallel transport along radial geodesics identifies the fibers and makes (\ref{intrinF}) scalar on the transported frame.

\medskip

\begin{lemma} \label{lemFonU}
    Let $\epsilon > 0$ and $z$ run in $Z^+_{\epsilon}$. Given $s \in \R_+,p \in \N$, there is
    $s_0 \in \R_+$ such that for any $\alpha \in \C$ with $\mathrm{Re}(\alpha) \geq s_0$,
    \begin{equation*} 
        \mathbf{F}^{\kappa}_{\alpha}(z)\in C_{\mathrm{loc}}^s(\U),
        \qquad
        \mathbf{F}^{\kappa}_{\alpha}(z)
        =
        \pazocal{O}_{C_{\mathrm{loc}}^s(\U)}\bigl(\langle z\rangle^{-p}\bigr).
    \end{equation*}
    Furthermore, we have
    \begin{equation*}
        \mathbf{F}^{\kappa}_{\alpha}(z) = \pazocal{O}_{\pazocal{D}'_{\Lambda_{\mathrm{F}}}(\U \setminus \Delta)}
    \bigl(\langle z\rangle^{-\infty}\bigr)
    \end{equation*}
    for all $\alpha \in \C$.
\end{lemma}

\begin{proof}
    This is the content of \cite[Lem.~6.1]{Dang:2020jch}, which remains valid for the pseudo-Riemannian family $\mathbf{F}^{\kappa}_{\alpha}(z)$, just as we highlighted it for the first properties that remain valid from the Lorentzian case in the proof of Proposition \ref{propQk}.
\end{proof}

\subsection{Transport equations and Hadamard parametrix} \label{transeq}

Let $E$ be equipped with a connection $\nabla^E$. We fix $x_0 \in M$, and we use normal coordinates centered at $x_0$. Denote by $V$ the Euler radial vector field. In a radial trivialization of $E$,
write
$$ \mathfrak h = V\bigl(\log |g|^{1/2}\bigr).$$
Let $u_{-1}=0$ and let
\[
    u_j \in C^\infty(\U;E\boxtimes E^*),
\]
called the \emph{Hadamard coefficients} of $P$, be defined recursively by
\begin{equation} \label{urecursion}
    2j u_j + \mathfrak h u_j + 2\nabla_V^E u_j + 2P u_{j-1} = 0,
    \qquad j \in \mathbb{N}_0,
\end{equation}
with
\[
    u_0(x_0,x_0)=\operatorname{Id}_{E_{x_0}}.
\]
The equations are ordinary differential equations along radial geodesics
and have unique smooth solutions after possibly shrinking $\U$. Using the trivialization with the parallel transport along radial geodesics introduced above, the system of transport equations can be written in terms of matrices.

Consider the truncated Hadamard parametrix
\begin{equation} \label{truncHada}
    H_N(z)
    =
    \chi \sum_{j=0}^N u_j \mathbf{F}_j(z) \in \pazocal{D}'(M \times M;E\boxtimes E^*),
\end{equation}
where
$$ \chi \in C^\infty(M\times M)$$
is properly supported, supported in $\U$, and equal to one on a smaller
neighborhood $\U_1$ of the diagonal.

We denote by $q = (x;\xi)$ the points in $T^*M$, and we define $\Lambda_{\chi}$ to be the \textit{truncation of the Feynman wavefront} $\Lambda_{\mathrm{F}}$ from (\ref{unirelation}) where we removed the neighborhood $\U_1$ near the diagonal in $M \times M$. Namely,
\begin{equation} \label{truncFeyn}
   \Lambda'_{\chi} := \Lambda'_{\mathrm{F}} \cap \pi^{-1 }\left( \operatorname{supp}\textnormal{d}\chi \right),
\end{equation}
where $\pi$ is the base projection, i.e.
$$\pi : S^*M \times S^*M \ni \left(\nu(x_1;\xi_1),\nu(x_2;\xi_2)\right) \mapsto (x_1,x_2) \in M \times M,$$
in the notation of Definition \ref{defprimerel}.

\smallskip

\begin{proposition} \label{propqnrn}
For every $N \in \N$, we have the equality
\begin{equation} \label{idequality}
    (P-z)H_N(z)
    =
    \lvert g \rvert^{-\frac{1}{2}} \delta_{\Delta}  \operatorname{Id}_E + \, q_N(z) + r_N(z),
\end{equation}
in the sense of distributions in $\pazocal{D}'(\pazocal{U}, E \boxtimes E^*)$, with
\begin{equation*}
    q_N(z)=\chi(Pu_N)\mathbf{F}_N(z),
    \qquad
    r_N(z)
    =
    [P,\chi]\sum_{j=0}^N u_j\mathbf{F}_j(z),
\end{equation*}
and where $\delta_{\Delta}$ is defined locally by pullback as $\delta_{\Delta} = G^* \delta_0$ and extended globally by a partition of unity.
Given $s \in \R_+, p \in \N$, $N$ may be chosen so that $q_N(z)$ has $C^s$ kernel
and is $O(\langle z\rangle^{-p})$ in the corresponding local Sobolev norms.
The kernel of $r_N(z)$ is supported in $\operatorname{supp}\textnormal{d}\chi$, and 
\begin{equation} \label{rN}
    r_N(z)
    =
    \pazocal O_{\pazocal{D}'_{\Lambda_{\chi}}(M\times M, E \boxtimes E^*)}
    \bigl(\langle z\rangle^{-\infty}\bigr)
\end{equation}
uniformly for $z\in\gamma_\varepsilon$.
\end{proposition}

\begin{proof}
The identities in normal coordinates are direct pseudo-Riemannian analogues
of \cite[Lem 5.6--Lem 5.7]{Dang:2020jch}. Note that $ \lvert g \rvert^{-\frac{1}{2}} \delta_{\Delta}(x,y)$ is the Schwartz kernel of the identity operator (we have to multiply by $\lvert g \rvert^{-\frac{1}{2}}$ since we integrate against the volume element).

Consider the untruncated Hadamard parametrix $H_N(z) =\sum_{j=0}^N u_j \mathbf{F}_j(z)$ for a moment. Section \cite[\S~5.3.1]{Dang:2020jch} proves, independently of the signature of the metric, that 
\begin{equation} \label{idequalityuntrunc}
    (P-z)H_N(z)
    =
    \lvert g \rvert^{-\frac{1}{2}} \delta_{\Delta}  \operatorname{Id}_E + (Pu_N)\mathbf{F}_N(z),
\end{equation}
as distributions in $\pazocal{D}'(\U,E \boxtimes E^*)$, by using the transport equations (\ref{urecursion}) and the above expression of the Hadamard parametrix. Note that the scalar computations in \cite[\S~5.3.1]{Dang:2020jch} are not affected by the bundle-setting considered here.

Now, considering (\ref{truncHada}), multiplying the parametrix by the cut-off $\chi$ yields an additional term in the above equality. This turns (\ref{idequalityuntrunc}) into 
$$(P-z)H_N(z)
    =
    \lvert g \rvert^{-\frac{1}{2}} \delta_{\Delta}  \operatorname{Id}_E + \, (Pu_N)\mathbf{F}_N(z)\chi + r_N(z),$$
where $r_N = [P,\chi]\sum_{j=0}^N u_j\mathbf{F}_j(z)$ as explained in the proof of \cite[Prop.~6.2]{Dang:2020jch}. This last term vanishes whenever $\chi=1$ or $\chi=0$, since it involves products with derivatives of the diagonal cut-off
function $\chi$. This means that $
r_N(z,x,y)=0$
for $(x,y)\in \U_1$, and also that
$
r_N(z,x,y)=0$
for $(x,y)\notin \U$. This implies that the Schwartz kernel
$
r_N(z)\in \pazocal{D}'(M\times M,E \boxtimes E^*)$
defines a proper operator. Now, observe that $r_N(z)$ is defined in terms of the distributions $$
\mathbf{F}_j(z)
=
\pazocal{O}_{\pazocal{D}'_{\Lambda_{\mathrm{F}}}(\U\setminus\Delta)}
\bigl(\langle z \rangle^{-\infty}\bigr)$$
by Lemma \ref{lemFonU}, which are restricted outside the diagonal. Therefore, $r_N(z)$ verifies
\[
r_N(z)
=
\pazocal{O}_{\pazocal{D}'_{\Lambda_\chi}(M\times M,E \boxtimes E^*)}
\bigl(\langle z \rangle^{-\infty}\bigr),
\]
uniformly for $z\in\gamma_\varepsilon$.
\end{proof}

\begin{remark}
\textnormal{In the scalar case, for $P=\Delta_g+V$, the potential is included in the term \(Pu_{j-1}\) in (\ref{urecursion}), while the model family remains \(F_j(z)\). For a constant mass, Dang--Wrochna instead use
$F_j(z-m^2)$ and mass-independent transport coefficients.} 
\end{remark}

\subsection{The Hadamard parametrix approximates the resolvent} \label{sectlocalexppseudoRcxpow}

Recall the definition of the complex powers. Namely, for $\Re (\alpha) > 0$, consider
\begin{equation*}(P -i\epsilon)^{-\alpha} = \frac{1}{2 \pi i } \int_{\gamma_{\epsilon}} (z -i\epsilon)^{-\alpha}(P - z)^{-1} dz.
\end{equation*}

Now, let $R(z)$ denote the resolvent $\resP$ on the complex contour $\gamma_{\epsilon}$. Formally, if one multiplies (\ref{idequality}) on the left by $R(z)$, it yields the equality
\begin{equation} \label{eqRHeterr}
     R(z) = H_N(z) - R(z)q_N(z) - R(z)r_N(z),
\end{equation}
in which the two error terms present different analytic behaviors. Ultimately, we would like to justify the insertion of equation (\ref{eqRHeterr}) in the representation of the complex powers. To this end, one has to characterize the analytic behaviors of the error terms appearing in (\ref{eqRHeterr}), uniformly in $z \in \gamma_{\varepsilon}$. On the one hand, the error $q_N(z)$ is controlled by taking $N$ large enough and by using global Sobolev estimates. On the other hand, the error $r_N(z)$ is controlled by the uniform wavefront relation (\ref{unirelation}). We want to understand whether these remainders could be as smooth as we want near the diagonal as in the Lorentzian case \cite[Prop.~6.3]{Dang:2020jch}, and, if not, under what conditions this might happen. We will make sense of the equality (\ref{eqRHeterr}) in Theorem \ref{thmRz}. The case of $q_N(z)$ is quite straightforward, as highlighted in the proof of Theorem \ref{thmRz} below. Hence we focus on the error $r_N(z)$ in the following discussion.
 
\subsubsection{The return relation and relative forward convexity} \label{returnSect}

 The error term $r_N(z)$, defined by 
 $$r_N(z)
    =
    [P,\chi]\sum_{j=0}^N u_j\mathbf{F}_j(z) \in \pazocal{D}'(M \times M,E \boxtimes E^*), $$
is supported in the cut-off set $\operatorname{supp} \textnormal{d}\chi$ by Proposition \ref{propqnrn}. The question is whether a singularity emitted from $\operatorname{supp} \textnormal{d}\chi$ can possibly return to the base diagonal. In the Lorentzian setting, under the assumption of at least \emph{strong causality}, we know that after choosing a sufficiently small causally convex normal neighborhood, any null geodesic with endpoints in such a neighborhood remains inside it. Hence a singularity generated in the cut-off set of the local Hadamard parametrix cannot propagate back to the region where the diagonal pullback is computed. Indeed, in a strongly causal Lorentzian manifold one can choose arbitrarily small convex normal neighborhoods that are causally convex in the ambient spacetime; see \cite[Cor.~2 and the discussion following it]{Minguzzi_2014}. Note that, in \cite{Dang:2020jch}, the Lorentzian manifold is assumed to be \emph{globally hyperbolic}, and is therefore strongly causal.

In general, a convex normal neighborhood, whose existence in every signature was recalled in Remark \ref{remconvexnei}, ensures the existence and uniqueness of the short geodesic used in the local Hadamard construction. It does not exclude a second null geodesic which leaves that neighborhood and later returns. Again, the main issue is whether a singularity emitted from $\textnormal{supp } \textnormal{d}\chi$ can return to a full neighborhood of the base diagonal. In mixed signature, the null cone does not in general determine a distinguished future/past decomposition, and hence there is no canonical notion of causal convexity with the same propagation property. Recent higher-signature notions of global hyperbolicity pursue a different geometric framework; see \cite{troubat2025globalhyperbolicityhighersignatures}.
The characteristic flow is still well-defined, but there is no canonical causal future preventing a bicharacteristic from leaving a local coordinate neighborhood and returning to it. 

We therefore replace causal convexity by an explicit geometric no-return condition formulated in terms of the Hamiltonian relation associated with the chosen Feynman orientation. This argument will allow us to perform the diagonal pullback of the complex powers in Theorem \ref{thmexpanresol} (2). 

Now, under an additional hypothesis that we call \textit{relative forward convexity} (\ref{defrelforconv}), we will show that the complex powers have a smooth Hadamard remainder in the neighborhood $\U$ of the diagonal in Theorem \ref{thmexpanresol} (1). This recovers the smooth-remainder expansion in the Lorentzian setting; see \cite[Lem.~7.1]{Dang:2020jch}. The restriction of the complex powers to the diagonal is thus straightforward in this case.

\begin{remark} [Causal convexity in strongly causal Lorentzian manifolds] \label{remstronglycausal}
    \textnormal{Consider a Lorentzian manifold assumed to be globally hyperbolic, or, more generally, strongly causal. In the setting of  \cite[Cor.~2]{Minguzzi_2014}, choose nested neighborhoods $ O \Subset O_1$, with $O$ causally convex and $O_1$ convex normal. Take $U_0 \subset O \times O$, and choose $\chi = 1$ on $O \times O$ with support in $O_1 \times O_1$. The cut-off set then lies outside $O \times O$, and a projected future-directed null bicharacteristic with endpoints in $O$ cannot pass through that set.}
\end{remark}

Let $\Lambda'_{\chi}$ denote the primed relation introduced in Definition \ref{defprimerel},
associated with the cone $\Lambda_{\chi}$ defined in (\ref{truncFeyn}).
Define the return relation
\begin{equation*}
    \pazocal{R}'_{\chi}
    :=
    \Lambda'_{\mathrm{F}} \circ \Lambda'_{\chi}.
\end{equation*}
Note that we can decompose the Feynman wavefront set into 
$$\Lambda'_{\mathrm{F}} =  C'_+ \,\cup \, D'_+$$
where $C_+$ and $D_+$ are subsets of $T^*(M\times M) \setminus o$ defined by 
$$C'_+ :=  \left\{(q_1,q_2) \in S^*M \times S^*M    \, \vert\,  q_1 \succ q_2 \right\}, \qquad D'_+ := \left\{(q,q) \in S^*M \times S^*M\right\},$$
in the sense of Definition \ref{defprimerel}.

\begin{definition}[Relative forward convexity] \label{defrelforconvdef}
Fix a smaller neighborhood $\U_0 \Subset \U_1$ of the diagonal.
The cut-off pair $(\U_0,\chi)$ is \emph{relatively forward convex} if
\begin{equation} \label{defrelforconv}
    \pazocal{R}'_{\chi} \cap \pi^{-1}(\U_0)
    =
    \varnothing,
\end{equation}
Thus no singularity generated in the cut-off set can propagate in the selected Hamilton direction back over $\U_0$.
\end{definition}

Consider $D'_+$ in the primed kernel notation. Avoidance of this set is precisely Hörmander's criterion for restricting a kernel to the base diagonal.

\begin{lemma}[Non-trapping implies diagonal transversality] \label{lemconv}
Under Hypothesis \ref{theH},
\begin{equation*} 
    \pazocal{R}'_{\chi} \cap D'_+
    =
    \varnothing.
\end{equation*}
Thus, diagonal transversality is a consequence of non-trapping.
\end{lemma}

\begin{proof}
First,
\[
    \Lambda'_{\chi} \cap D’_+
    =
    \varnothing,
\]
since $\textnormal{d}\chi$ vanishes on $\U_1$ and hence the base support of
$\Lambda'_{\chi}$ is disjoint from the diagonal. Suppose now that
\[
    (q,q) \in C'_+ \circ \Lambda'_{\chi}.
\]
Then there is a phase point $q_1$ and positive times $s,t>0$ such that
\[
    q_1 = \Phi_s(q),
    \qquad
    q = \Phi_t(q_1).
\]
The times are strictly positive because the cut-off relation lies off the
diagonal. Consequently
\[
    \Phi_{s+t}(q)=q,
\]
so the null bicharacteristic through $q$ is periodic and hence trapped. Indeed, it remains in a compact subset of the interior fiber infinity $S^*M$ and cannot tend to either radial set \(L_+\) or \(L_-\), that lie within $\partial {}^{\mathrm{sc}}T^* \Mb$, which is base infinity in the scattering cotangent bundle. 
This contradicts Hypothesis \ref{theH}.
\end{proof}

\begin{remark}
    \textnormal{Definition \ref{Pnontrap} about non-trapping does include bicharacteristics staying within $L_+$ or $L_-$. However, such an option does not save the null bicharacteristic considered in the above proof, because that case only applies if the initial point is already in the radial set $L_{\pm} \subset \partial {}^{\mathrm{sc}}T^* \Mb$, which is at base infinity.}
\end{remark}

To lighten the notation, we omit the bundle \(E\boxtimes E^*\) from the notation whenever it is clear from the context. Thus, we write \(\pazocal D'(\U)\) in place of \(\pazocal D'(\U;E\boxtimes E^*)\), and $C^s(\U)$ instead of $C^s(\U;E\boxtimes E^*)$.

\begin{theorem} \label{thmRz}
Assume Hypothesis \ref{theH}. For every $s \in \R_+,p \in \N$, one can choose $N$ so that, uniformly on $\gamma_\varepsilon$,
    \begin{equation*}
        R(z)
        =
        \chi \sum_{k=0}^{N} u_k F_k(z) + E_N(z), \qquad E_N(z) = -R(z)r_N(z) -R(z)q_N(z),
\end{equation*}
in the sense of operators $\pazocal{O}_{C^{\infty} \to C^{\infty}}(\langle z \rangle^{-1})$, and the second error $ E_{N,2}(z) = - R(z)q_N(z)$ satisfies $$E_{N,2}(z) = \pazocal{O}_{C^s(\U)}\bigl(\langle z\rangle^{-p}\bigr).$$
If $(\U_0,\chi)$ is relatively forward convex, the first error term $E_{N,1}(z) = -R(z)r_N(z)$, and hence the whole error $E_N(z)$, verify
     $$ E_N(z)
        =
        \pazocal{O}_{C^s(\U_0)}\bigl(\langle z\rangle^{-p}\bigr).$$

\end{theorem}

\begin{proof} 
Let $s \in \R_+$, $p \in \N$, and set
\[
    E_{N,1}(z)=-R(z)r_N(z) ,  \quad
    E_{N,2}(z)=-R(z)q_N(z),
\]
where we put $E_N(z) = E_{N,1}(z) + E_{N,2}(z)$. First, recall that 
$R(z) = \pazocal{O}_{C^{\infty} \to C^{\infty}}(\langle z \rangle^{-1})$
along $\gamma_\varepsilon$ by Proposition \ref{propfredreal}. Now, we use that Lemma \ref{lemFonU} concerning $\mathbf{F}_{\alpha}(z)$ makes $q_N = \chi(Pu_N)\mathbf{F}_N(z)$ arbitrarily regular and rapidly decreasing in
$z$ as $N$ grows, i.e. $q_N = \pazocal{O}_{C^s(\U)}\bigl(\langle z\rangle^{-p}\bigr)$. The composition $E_{N,2}(z) = -R(z)q_N(z)$ makes sense by the operator wavefront calculus, i.e. Lemma \ref{lemcompo}, and the global estimate in
\cite[Prop.~3.6]{dang2024diracoperatorslocalinvariants} then gives the required regularity and decay, namely $E_{N,2}(z)
        =
        \pazocal{O}_{C^s(\U)}\bigl(\langle z\rangle^{-p}\bigr)$.

For $E_{N,1}(z) = -R(z)r_N(z)$, first observe that we have $\wf'^{(s)}_{\langle z \rangle^{-\infty}}\left(r_N(z)\right) \subset \Lambda'_{\chi}$ by Lemma \ref{lemusforrN}, since we know that $r_N(z) = \pazocal{O}_{\pazocal{D}'_{\Lambda_\chi}}
\left(\langle z \rangle^{-\infty}\right)$ by (\ref{rN}). Then, together with $\wf'^{(s)}_{\langle z \rangle^{-1}}\left(R(z)\right) \subset \Lambda'_{F}$ by relation (\ref{unirelation}), Lemma \ref{lemcompo} yields
\begin{equation*} 
    \wf'^{(s)}_{\langle z \rangle^{-\infty}}\left(E_{N,1}(z)\right)
    \subset
    \pazocal{R}'_{\chi} = \Lambda'_{\mathrm{F}} \circ \Lambda'_{\chi}.
\end{equation*}
Now, identity (\ref{idequality}) reads, written here in the sense of operators, 
\begin{equation*} 
    (P-z) H_N(z)
    =
    \Id_E + \, q_N(z) + r_N(z).
\end{equation*}
One can legitimately multiply on the left by the resolvent, which yields
$$ R(z)
        =
        \chi \sum_{k=0}^{N} u_k F_k(z) + E_N(z),$$
in the sense of operators $\pazocal{O}_{C^{\infty} \to C^{\infty}}(\langle z \rangle^{-1})$.

If $(\U_0,\chi)$ is relatively forward convex, every phase point over $\U_0$ is disjoint from the return relation by Definition \ref{defrelforconvdef}, namely $$\pazocal{R}'_{\chi} \cap \pi^{-1}(\U_0)
    =
    \varnothing.$$ Knowing that $\wf'^{(s)}_{\langle z \rangle^{-\infty}}\left(E_{N,1}(z)\right)
    \subset
    \pazocal{R}'_{\chi}$, it thus proves that $E_{N,1}(z) = \pazocal{O}_{C^s(\U_0)}\bigl(\langle z\rangle^{-p}\bigr)$. Combining the two errors shows that, under relative forward convexity, we have 
$$R(z)
        =
        \chi \sum_{k=0}^{N} u_k F_k(z) + E_N(z), \qquad E_N(z) = \pazocal{O}_{C^s(\U_0)}\bigl(\langle z\rangle^{-p}\bigr),$$
which holds in the sense of distributions in $\pazocal{D}'(\U_0)$.
\end{proof}

\begin{remark} [Non-trapping]
    \textnormal{As we already said at the beginning of the section, non-trapping excludes the return to the same \textit{phase point}, which constitutes Lemma \ref{lemconv}. However, it does not exclude the possibility of a return to the same or nearby base point with a different covector, which would obstruct the near-diagonal smoothness of the remainder (while not actually affecting the part responsible for the local residues).}
\end{remark}

\begin{remark}[Pseudoconvexity and disprisonment]
\upshape
The classical notions of geodesic pseudoconvexity and disprisonment provide a natural framework for controlling excursions of bicharacteristics and global solvability; see
\cite{BeemParker1989,DelRiegoParker1995}. Recall that geodesic
pseudoconvexity says, roughly, that for every compact set \(K\subset M\) there
exists a larger compact set \(K'\subset M\) such that every geodesic segment
with endpoints in \(K\) is contained in \(K'\). Disprisonment, on the other
hand, excludes inextendible geodesics entirely contained in a compact subset of
\(M\).

These conditions are not the same as the relative forward convexity condition
\eqref{defrelforconv}. First, pseudoconvexity and disprisonment are global
conditions on the base projections of geodesics, whereas
\eqref{defrelforconv} is a geometric condition on the Hamilton relation in
phase space. Second, relative forward convexity is tied to the chosen
Feynman orientation: it only forbids propagation in the selected forward
Hamilton direction. 
\end{remark}

\begin{example}[Exact pseudo-Euclidean model]
\textnormal{Let $ p_0(\xi)=Q_\kappa(\xi)$
be the exact pseudo-Euclidean principal symbol. Since \(p_0\) is independent of
the base variable, the Hamilton vector field is
$$
    H_{p_0}=(\partial_\xi p_0)\cdot \partial_x, $$
and the covector \(\xi\) is constant along bicharacteristics. Thus, if $v=\partial_\xi p_0(\xi)$ the projected null bicharacteristics are affine lines.}

\textnormal{Choose an auxiliary Euclidean norm \(|\cdot|\), and let the cut-off $\chi$ be radial in
the difference variable. More precisely, choose \(0<r_1<r_2\) such that
\[
    \operatorname{supp} d\chi
    \subset
    \{(z,y): r_1\leq |z-y|\leq r_2\},
    \qquad
    U_0\subset \{(x,y): |x-y|<r_1\}.
\]
We claim that the pair \((U_0,\chi)\) is relatively forward convex. Indeed, suppose that a composed forward relation passes from $y$ to $z$
and then from $z$ to $x$. Since the canonical relation is composed with
matching intermediate covector, the two segments lie on the same projected
bicharacteristic. Hence there exist \(s,t\geq0\) and a constant null velocity
\(v=\partial_\xi p_0(\xi)\) such that
\[
    z-y=t v,
    \qquad
    x-z=s v.
\]
Since the first segment meets $\operatorname{supp} d\chi$, we have
\[
    |t v|=|z-y|\geq r_1,
\]
while $
    x-y=(s+t)v$
verifies
\[
    |x-y|=|(s+t)v|=(s+t)|v|\geq t|v|=|z-y|\geq r_1.
\]
This contradicts \((x,y)\in U_0\), for which \(|x-y|<r_1\). Thus no singularity generated in $\operatorname{supp} d\chi$ can propagate in the selected forward Hamilton direction back over \(\U_0\), and the cut-off pair is relatively forward convex.}
\end{example}

\subsection{Local expansion of the pseudo-Riemannian complex powers} \label{sectlocexpcxpow}

Recall the definition of the complex powers, namely 
\begin{equation} \label{cxcontourint} (P -i\epsilon)^{-\alpha} = \frac{1}{2 \pi i } \int_{\gamma_{\epsilon}} (z -i\epsilon)^{-\alpha}(P - z)^{-1} dz
\end{equation}
for $\mathrm{Re}(\alpha) > 0$. We first prove a preliminary lemma about the Gamma function that will be used throughout the section.

\begin{lemma} \label{lemgamma}
    For any $k \in \N_0$, $\alpha \in \C \setminus \Z$, the Gamma function $\Gamma$ verifies 
    $$\frac{1}{\Gamma(\alpha+k)} = (-1)^k \frac{\Gamma(1 - \alpha - k)}{\Gamma(\alpha) \Gamma(1- \alpha) }.$$
\end{lemma}

\begin{proof}
   Euler's reflection formula says 
\begin{equation*} \label{eulerformula}
    \Gamma(a)\Gamma(1- a) = \frac{\pi}{\sin(\pi a)}
\end{equation*}
for any $a \notin \Z$. Also,
$$\Gamma(\alpha+k) \Gamma(1 - \alpha - k) = \frac{\pi}{\sin(\pi (\alpha + k))},$$
where $\sin(\pi (\alpha+k)) = (-1)^k \sin(\pi \alpha)$, hence  
$$\frac{1}{\Gamma(\alpha+k)} = \frac{\Gamma(1 - \alpha - k) \sin(\pi (\alpha+k))}{\pi} = (-1)^k \frac{\Gamma(1 - \alpha - k) \sin(\pi \alpha)}{\pi}.$$
Using again the reflection formula, we get 
$$\frac{1}{\Gamma(\alpha+k)} = (-1)^k \frac{\Gamma(1 - \alpha - k)}{\Gamma(\alpha) \Gamma(1- \alpha) }.$$
\end{proof}

Recall that any distribution $K \in \pazocal{D}'(M \times M, E \boxtimes E^*)$ that has a well-defined diagonal pullback  defines a single-variable distribution $\iota_\Delta^* K \in \pazocal{D}'\left(M, \mathrm{End}(E)\right)$; see relation (\ref{relpullbackdiag}). Recall also that we omit the bundle $E \boxtimes E^*$ from the notation.

The following theorem provides a local expansion of the complex powers which, after restriction to the diagonal, allows us to compute their residues.

\begin{theorem} \label{thmexpanresol}
    Assume Hypothesis \ref{theH}. For every $s \in \R_+$, $p \in \N$ there exists $N \in \N$ such that the following properties hold.
    \begin{enumerate} 
        \item If $(\U_0,\chi)$ is relatively forward convex, then 
        \begin{equation} \label{smoothexpcxpowers}
            \Pa = \frac{1}{\Gamma(\alpha)} \sum_{k=0}^{N} u_k \mathbf{F}_{k + \alpha - 1}(i\epsilon) + R_N(i\epsilon,\alpha)
        \end{equation}
        as distributions on $\U_0$, where $R_N(i\epsilon,\alpha)$ is in $C^s(\U_0)$ and holomorphic for $\Re(\alpha) > -p$.
        \item Without relative forward convexity, the diagonal pullback exists  from the initial half-plane where the diagonal pullback of the integrated Hadamard terms is defined and verifies 

\begin{equation} \label{expcxpowerson diag}
    \iota_\Delta^*(P-i\varepsilon)^{-\alpha}
    =
    \frac{1}{\Gamma(\alpha)}
    \sum_{k=0}^{N}
    \iota_\Delta^*
    \left(
        u_k \mathbf{F}_{\alpha+k-1}(i\varepsilon)
    \right)
    +
    \rho_N(i\varepsilon,\alpha) \in \pazocal{D}'\left(M, \mathrm{End}(E)\right),
\end{equation}
where $\rho_N(i\varepsilon,\alpha)$ is holomorphic for
$\operatorname{Re}(\alpha)>-p$ with values in any prescribed local
distributional Sobolev space.

The analogous lower half-plane statements are obtained by replacing
$i\varepsilon$ by $-i\varepsilon$.
    \end{enumerate}
\end{theorem}

In any case, the on-diagonal restriction of the complex powers is well-defined, and its meromorphic continuation, proven in the next section, will be denoted by 
\begin{equation} \label{defzetafunc}
   \alpha \mapsto \zeta^\kappa_{g,\varepsilon}(\alpha) := \iota_\Delta^*(P-i\varepsilon)^{-\alpha} \in \pazocal{D}'\left(M;\mathrm{End}(E)\right) 
\end{equation}
and will be called the \emph{pseudo-Riemannian spectral zeta function density} of $(M, g)$ of signature $(\kappa,n-\kappa)$. 

\begin{proof}
    We use the relation between the resolvent $R(z)$ and the Hadamard parametrix $H_N(z)$ proven in Theorem \ref{thmRz}. One can insert it in the complex powers representation $(\ref{cxcontourint})$ using the decay along $\gamma_{\epsilon}$ of the different terms appearing in Theorem \ref{thmRz}. For $\Re(\alpha) > 0$, it yields
\begin{align*} \label{messs}
    (P -i\epsilon)^{-\alpha} 
            &= \frac{1}{2 \pi i }  \int_{\gamma_{\epsilon}} (z -i\epsilon)^{-\alpha} R(z) \,dz  \\
            &= \chi \sum_{k = 0}^{N} u_k  \,  \frac{1}{2 \pi i } \int_{\gamma_{\epsilon}} (z -i\epsilon)^{-\alpha} \textbf{F}_k(z)\, dz  + \frac{1}{2 \pi i} \int_{\gamma_{\epsilon}} (z -i\epsilon)^{-\alpha} E_{N}(z) \,  dz.
\end{align*}
Now, by Lemma 4.7 in \cite{Dang:2020jch} (which holds for any signature) and continuity of the pullback by $G$, we have 
\begin{equation*} \label{lemma47}
 \frac{1}{2 \pi i } \int_{\gamma_{\epsilon}} (z -i\epsilon)^{-\alpha} \textbf{F}_k(z) \, dz = \frac{(-1)^k \Gamma(-\alpha + 1)}{\Gamma(-\alpha - k + 1)\Gamma(\alpha + k) } \textbf{F}_{k + \alpha - 1}(i\epsilon).
\end{equation*} 

By Lemma \ref{lemgamma}, the coefficients simplify to
\begin{equation*} 
    \frac{1}{\Gamma(\alpha)} = (-1)^k \frac{\Gamma(1 - \alpha)}{\Gamma(1-\alpha-k) \Gamma(\alpha+k)}.
\end{equation*}
Therefore, the complex powers can be expressed as 
\begin{equation} \label{eqcxpowersinsideproof}
    (P -i\epsilon)^{-\alpha} = \frac{\chi}{\Gamma(\alpha)} \sum_{k=0}^{N} u_k \mathbf{F}_{k + \alpha - 1}(i\epsilon) + R_N(i\epsilon,\alpha).
\end{equation}
This equality is understood as operators in $\pazocal B(H_c^*,H_{\loc}^{*+r})$ for some $r \in \R$, or as distributions in $\pazocal{D}'(\U)$, where $$R_N(i\epsilon,\alpha) := \frac{1}{2 \pi i} \int_{\gamma_{\epsilon}} (z -i\epsilon)^{-\alpha} E_{N}(z) \,  dz.$$ 

Under relative forward convexity, given $s,p$, we have for $N$ sufficiently large that $E_N(z)
        =
        \pazocal{O}_{C^s(\U_0)}\bigl(\langle z\rangle^{-p-2}\bigr)$ by Theorem \ref{thmRz} . The $\pazocal{O}(\langle z \rangle^{-p-2})$ estimate for $E_N(z)$ from Theorem \ref{thmRz} justifies that we can take the contour integral and it moves the holomorphy half-plane arbitrarily far left as $N$ increases. Together with the uniform regularity of $E_N(z)$ it finally proves that $R_N(i\epsilon,\alpha)$ lies in $C^s(\U_0)$ and is holomorphic for $\Re(\alpha) > - p$.
    
    Without relative forward convexity, we start from the operator equality (\ref{eqcxpowersinsideproof}). First, given $s,p$, we have for $\Re(\alpha)$ large enough that $\mathbf{F}_{k + \alpha - 1}(i\epsilon) \in C^s(\U)$ by Lemma \ref{lemFonU}. Therefore the pullback of $u_k \mathbf{F}_{k + \alpha - 1}(i\epsilon)$ is trivially well-defined, i.e. $$ x \mapsto \iota_\Delta^*(u_k \mathbf{F}_{k + \alpha - 1}(i\epsilon))(x) = u_k(x,x)\mathbf{F}_{k + \alpha - 1}(i\epsilon)(x,x) \in C^s_{\loc}\left(M;\mathrm{End}(E)\right),$$ as $C^s$-sections. Note that $\iota_\Delta^*(u_k)$ itself constitutes an element of $C^{\infty}\left(M;\mathrm{End}(E)\right)$.
    
    Then, knowing that $ \wf'^{(s)}_{\langle z \rangle^{-\infty}}\left(E_{N,1}(z)\right)
    \subset
    \pazocal{R}'_{\chi}$ and that $E_{N,2}(z)
        =
        \pazocal{O}_{C^s(\U)}\bigl(\langle z\rangle^{-p-2}\bigr)$, we now determine the wavefront set of $R_N(i\epsilon,\alpha)$.
        
    Let $s \in \R_+$, $p \in \N$, and let \((q_1,q_2)\notin \pazocal{R}'_{\chi}\). By definition of the operator wavefront set, there exist \(C_i\in\Psi^0(M)\), elliptic at \(q_i\), such that, for $N$ taken sufficiently large,
\[
    C_1E_N(z)C_2^*
    =
    \pazocal O_{H^*\to H^{*+s}}\bigl(\langle z\rangle^{-p-2}\bigr)
\]
uniformly for \(z\in\gamma_\varepsilon\), for $N$ taken large enough.  Hence
\[
    z \, \mapsto
    (z-i\varepsilon)^{-\alpha} C_1E_N(z)C_2^*
\]
is integrable along \(\gamma_\varepsilon\) with values in the Banach space $\pazocal B(H_{\mathrm c}^\ell,H_{\loc}^{\ell+s})$ for any $\ell \in \R$.
Thus the following is a Bochner integral in that space:
\[
    \frac{1}{2\pi i}
    \int_{\gamma_\varepsilon}
    (z-i\varepsilon)^{-\alpha}
 C_1E_N(z)C_2^*\,dz .
\]

Because \(C_1\), \(C_2^*\) are fixed continuous operators,
independent of \(z\), they may be inserted into the contour integral. Hence

\begin{equation} \label{eqinsertc_i}
    C_1R_N(i\varepsilon,\alpha)C_2^*
    =
    \frac{1}{2\pi i}
    \int_{\gamma_\varepsilon}
    (z-i\varepsilon)^{-\alpha}  C_1E_N(z)C_2^*\,dz
\end{equation}

in \(\pazocal B(H_{\mathrm{c}}^\ell,H_{\loc}^{\ell+s})\). Therefore, since $C_1E_N(z)C_2^* = \pazocal O_{H^*\to H^{*+s}}\bigl(\langle z\rangle^{-p-2}\bigr)$, from equality (\ref{eqinsertc_i}) we obtain 

\[
    C_1R_N(i\varepsilon,\alpha)C_2^*
    =
    \pazocal O_{H^*\to H^{*+s}}(1).
\]
By definition, it yields
\[
    (q_1,q_2)\notin
    \wf'^{(s)}\bigl(R_N(i\varepsilon,\alpha)\bigr).
\]

The above statement holds for any $s \in \R_+$, as well as for arbitrary $(q_1,q_2)\notin \pazocal{R}'_{\chi}$.  

\smallskip

Therefore, we just proved that
    $$\wf'\left(R_N(i\epsilon,\alpha)\right) \subset
    \pazocal{R}'_{\chi}$$
for the usual distributional wavefront. Finally,  Lemma \ref{lemconv} says that $$\pazocal{R}'_{\chi} \cap D'_+ = \varnothing, \qquad D'_+ = \left\{(q,q) \, \vert\,  q \in S^*M \right\},$$
which is, as we pointed out already, Hörmander's criterion for restricting a kernel to the base diagonal. Hence it allows us to apply Hörmander's pullback theorem to $R_N(i\epsilon,\alpha)$, considered as a distribution. It yields $\rho_N(i\epsilon,\alpha) := \iota_\Delta^*R_N(i\epsilon,\alpha)$, a well-defined distribution in $x \in M$ which is holomorphic for
$\operatorname{Re}(\alpha)>-p$.

 Combining the two pullbacks defines $\iota_\Delta^*\Pa$ as a distribution in $x \in M$ by the equality
 \begin{equation*} 
    \iota_\Delta^*(P-i\varepsilon)^{-\alpha}
    =
    \frac{1}{\Gamma(\alpha)}
    \sum_{k=0}^{N}
    \iota_\Delta^*
    \left(
        u_k \mathbf{F}_{\alpha+k-1}(i\varepsilon)
    \right)
    +
    \rho_N(i\varepsilon,\alpha) \in \pazocal{D}'(M;\mathrm{End}(E)).
\end{equation*}

The lower half-plane statements proceed analogously. 
\end{proof} 

\subsection{Residues of the pseudo-Riemannian spectral zeta function} \label{sectresidues}
 The analysis naturally separates according to the parity of \(n\). Consider, whether $n$ is even or odd, 
\begin{equation} \label{poles}
    \pazocal{P}_n = \left\{ \frac{n}{2} - k \ \vert \ k \in \N_0, \ \frac{n}{2} - k \notin - \N_0 \right\}.
\end{equation}

Therefore, 
\[
\pazocal{P}_n =
\begin{cases}
\left\{ \tfrac{n}{2},\, \tfrac{n}{2}-1,\, \ldots,\, 1 \right\},
& n \text{ even}, \\[4pt]
\left\{ \tfrac{n}{2},\, \tfrac{n}{2}-1,\, \tfrac{n}{2}-2,\, \ldots \right\},
& n \text{ odd}.
\end{cases}
\]
We will prove that the pole set of the meromorphically continued spectral zeta function is contained in $\pazocal{P}_n$.
To begin with, we study the on-diagonal restriction of the elementary family $\mathbf{F}^{\kappa}_{\beta}(z)$. Given any $s \in \R_+$, for $\Re(\beta)$ sufficiently large, the equality $$F^\kappa_\beta(z,0) = \mathbf{F}^\kappa_\beta(z,x,x)$$ is well-defined since $x \mapsto \mathbf F^\kappa_\beta(z,x)$ belongs  to $C^s_{\loc}(\R^n)$ by Lemma \ref{lemFonU}.

\begin{proposition} \label{propkfolwick}
Assume \(\operatorname{Re}(\beta)>n/2-1\).

For \(\operatorname{Im}(w)>0\),
\begin{equation} \label{Fimwpos}
    F^\kappa_\beta(w,0)
    =
    \frac{i^\kappa}{(4\pi)^{n/2}}
    \Gamma\!\left(\beta+1-\frac n2\right)
    (-w)^{\frac n2-\beta-1}.
\end{equation}
For \(\operatorname{Im}(w)<0\),
\begin{equation} \label{Fimwneg}
    F^\kappa_\beta(w,0)
    =
    \frac{(-i)^\kappa}{(4\pi)^{n/2}}
    \Gamma\!\left(\beta+1-\frac n2\right)
    (-w)^{\frac n2-\beta-1}.
\end{equation}
Both identities extend meromorphically in \(\beta \in \C\).
\end{proposition}

\begin{proof}
Recall the expression of the distributions $F_{\beta}(z)$, namely
\begin{equation*} 
    F_\beta^\kappa(w,x)
    :=
    \frac{\Gamma(\beta+1)}{(2\pi)^n}
    \int_{\mathbb{R}^n}
    e^{i\langle x,\xi\rangle}
    \bigl(Q_\kappa(\xi)-w\bigr)^{-\beta-1}
    \,d^n\xi.
\end{equation*}
Put $
    \lambda:=\beta+1$, and $ \nu:=\lambda-\frac n2 = \beta+1-\frac n2$. 
We first assume $
    \operatorname{Re}(\nu)>0$, i.e. $
    \operatorname{Re}(\beta)>\frac n2-1$. 
    
    In mixed signature, the formal integral written as
\[
    \int_{\mathbb R^n}
    (Q_\kappa(\xi)-w)^{-\lambda}\,d\xi
\]
is not absolutely convergent in general. Hence we use the Schwinger parameter formula
and the Fresnel integral for the quadratic form \(Q_\kappa\).

Assume first that \(\operatorname{Im}(w)>0\). Since \(Q_\kappa(\xi)-w\) lies in
the lower half-plane, we have
\[
    (Q_\kappa(\xi)-w)^{-\lambda}
    =
    \frac{e^{i\pi\lambda/2}}{\Gamma(\lambda)}
    \int_0^\infty
    e^{-iu(Q_\kappa(\xi)-w)}
    u^{\lambda-1}\,du .
\]
Using the normalization in the definition of \(F^\kappa_\beta\), this gives
\[
\begin{aligned}
    F^\kappa_\beta(w,x)
    &=
    \frac{e^{i\pi\lambda/2}}{(2\pi)^n}
    \int_0^\infty
    e^{iuw}u^{\lambda-1}
    \left(
        \int_{\mathbb R^n}
        e^{i\langle x,\xi\rangle}
        e^{-iuQ_\kappa(\xi)}\,d\xi
    \right)du .
\end{aligned}
\]
Here the inner integral is understood as an oscillatory integral. Its existence is justified, for instance, by inserting the regularization \(e^{-\delta|\xi|^2}\) and then letting
\(\delta \rightarrow 0^+\). This also justifies the use of Fubini's theorem for interchanging the order of integrations.

Now, write
\[
    Q_\kappa(\xi)
    =
    -|\xi_-|^2+|\xi_+|^2,
    \qquad
    \xi_-=(\xi_1,\ldots,\xi_\kappa),
    \qquad
    \xi_+=(\xi_{\kappa+1},\ldots,\xi_n).
\]
The Fresnel formula gives, for \(u>0\),
\[
    \int_{\mathbb R^n}
    e^{i\langle x,\xi\rangle}
    e^{-iuQ_\kappa(\xi)}\,d\xi
    =
    e^{i\pi(2\kappa-n)/4}
    \pi^{n/2}u^{-n/2}
    e^{iQ_\kappa(x)/(4u)} .
\]
Therefore
\[
    F^\kappa_\beta(w,x)
    =
    \frac{\pi^{n/2}}{(2\pi)^n}
    e^{i\pi\lambda/2}
    e^{i\pi(2\kappa-n)/4}
    \int_0^\infty
    e^{iuw}
    e^{iQ_\kappa(x)/(4u)}
    u^{\nu-1}\,du .
\]
At \(x=0\), the factor \(e^{iQ_\kappa(x)/(4u)}\) is equal to \(1\). Hence
\[
    F^\kappa_\beta(w,0)
    =
    \frac{\pi^{n/2}}{(2\pi)^n}
    e^{i\pi\lambda/2}
    e^{i\pi(2\kappa-n)/4}
    \int_0^\infty
    e^{iuw}u^{\nu-1}\,du .
\]
The last integral is convergent at infinity since \(\operatorname{Im}(w)>0\),
and it is convergent at \(u=0\) precisely when \(\operatorname{Re}(\nu)>0\). The standard formula gives
\[
    \int_0^\infty e^{iuw}u^{\nu-1}\,du
    =
    \Gamma(\nu)(-iw)^{-\nu}.
\]
For \(\operatorname{Im}(w)>0\), one has $
    (-iw)^{-\nu}
    =
    e^{-i\pi\nu/2}(-w)^{-\nu}$.
Consequently,
\[
\begin{aligned}
    F^\kappa_\beta(w,0)
    &=
    \frac{1}{(4\pi)^{n/2}}
    e^{i\pi\lambda/2}
    e^{i\pi(2\kappa-n)/4}
    e^{-i\pi\nu/2}
    \Gamma(\nu)(-w)^{-\nu}  \\
    &=
    \frac{i^\kappa}{(4\pi)^{n/2}}
    \Gamma\!\left(\beta+1-\frac n2\right)
    (-w)^{\frac n2-\beta-1}.
\end{aligned}
\]
This proves the upper-half-plane formula.

If \(\operatorname{Im}(w)<0\), one uses the corresponding Fresnel formula and analogous computations yield
\[
\begin{aligned}
    F^\kappa_\beta(w,0)
    &=
    \frac{(-i)^\kappa}{(4\pi)^{n/2}}
    \Gamma\!\left(\beta+1-\frac n2\right)
    (-w)^{\frac n2-\beta-1}.
\end{aligned}
\]

The formulas have been proved in the half-plane $
    \operatorname{Re}(\beta)>\frac n2-1$.
Since the right-hand sides are meromorphic in \(\beta \in \C\), they give the meromorphic continuation of $\beta \mapsto F^\kappa_\beta(w,0)$.
\end{proof}

Recall the expression of the spectral zeta function density from Theorem \ref{thmexpanresol}, namely 
\begin{equation} \label{rappelzetadecomp}
   \zeta^\kappa_{g,\varepsilon}(\alpha) := \iota_\Delta^*(P-i\varepsilon)^{-\alpha} =
    \frac{1}{\Gamma(\alpha)}
    \sum_{j=0}^{N}
    \iota_\Delta^*
    \left(
        u_j \mathbf{F}^\kappa_{\alpha+j-1}(i\varepsilon)
    \right)
    +
    \rho_N(i\varepsilon,\alpha)  
\end{equation}
for $\Re(\alpha) > \frac{n}{2}$ and $N$ taken large enough, where $\alpha \mapsto \rho_N(i\varepsilon,\alpha)$ is holomorphic. Its finite sum part, the one relevant for the residues, is thus given by
$$ \frac{1}{\Gamma(\alpha)} \sum_{j=0}^{N} \iota_\Delta^*(u_j) F^\kappa_{j + \alpha - 1}(i\epsilon,0), \quad x \mapsto \iota_\Delta^*(u_j)(x) = u_j(x,x) \in C^{\infty}\left(M;\mathrm{End}(E)\right).$$
Now, insert \(\beta=\alpha+j-1\) in (\ref{Fimwpos}) from Proposition \ref{propkfolwick}, with $w = i \epsilon$. This gives the value of $F_{j + \alpha - 1}(i\epsilon,0)$. At $x \in M$, the upper-half-plane diagonal model term appearing in the finite sum is then 
\begin{equation} \label{diagmodel}
    \frac{i^\kappa u_j(x,x)
    \Gamma\!\left(\alpha+j-\frac n2\right)}
    {(4\pi)^{n/2}\Gamma(\alpha)}
    (-i\varepsilon)^{\frac n2-\alpha-j}.
\end{equation}
The candidate poles are therefore $\alpha=\frac n2-k$ for $k \in \N_0$.
The zeros of \(1/\Gamma(\alpha)\) remove precisely the candidates that are non-positive integers.

For completeness, let $k \in \N_0$ and put $\alpha_k=\frac n2-k$. We choose $N \geq k$ large enough so that the remainder $\rho_N$ in (\ref{rappelzetadecomp}) is holomorphic near $\alpha_k$. Then, at $\alpha = \alpha_k$ in (\ref{diagmodel}), the Gamma function
\[
    \Gamma\!\left(\alpha+j-\frac n2\right)
    =
    \Gamma(j-k+\alpha-\alpha_k)
\]
has residue $
    \frac{(-1)^{k-j}}{(k-j)!}$ for \(j\leq k\),
whereas for \(j>k\) it is regular. The factor \(1/\Gamma(\alpha_k)\) is
non-zero exactly when \(\alpha_k\notin -\mathbb N_0\). If \(n\) is odd, every \(\alpha_k\) is a half-integer, so no
cancellation occurs. If \(n\) is even, \(\alpha_k\) becomes a non-positive
integer at \(k=n/2\), and every later candidate is cancelled. 

This proves that $\alpha \mapsto \zeta^\kappa_{g,\varepsilon}(\alpha)$ admits meromorphic continuation, and that 
$$\pazocal{P}_n = \left\{ \frac{n}{2} - k \ \vert \ k \in \N_0, \ \frac{n}{2} - k \notin - \N_0 \right\}$$ constitutes its pole set without any additional microlocal argument. Hence the residue at $\alpha = \alpha_k$ of the spectral zeta density is well-defined as
\begin{equation} \label{eqresid}
    \operatorname*{res}_{\alpha=\alpha_k}\zeta^\kappa_{g,\varepsilon}(\alpha) 
    =
    \sum_{j=0}^{k}
    \iota_\Delta^*(u_j)
    \operatorname*{res}_{\alpha=\alpha_k}
    \left[
        \frac{1}{\Gamma(\alpha)}
        F^\kappa_{\alpha+j-1}(i\varepsilon,0)
    \right].
\end{equation}
   It is a smooth section of $\mathrm{End}(E)$, namely 
    $\operatorname*{res}_{\alpha=\alpha_k}\zeta^\kappa_{g,\varepsilon}(\alpha)  \in C^{\infty}\left(M;\mathrm{End}(E)\right)$.

In the following results, by an abuse of notation, we will write $\operatorname*{res}_{\alpha=\alpha_k}
    (P-i\varepsilon)^{-\alpha}(x,x)$ instead of writing $\left(\operatorname*{res}_{\alpha=\alpha_k}\zeta^\kappa_{g,\varepsilon}(\alpha)\right)(x)$, which is well-defined as an element of $\mathrm{End}(E_x)$.

\begin{theorem} \label{thmresbeforelim}
Let $k \in \N_0$ be such that $ \alpha_k=\frac n2-k\in\pazocal{P}_n $.
Under the hypotheses that make the diagonal expansion in Theorem \ref{thmexpanresol} valid,
\begin{equation} \label{resnolimminus}
    \operatorname*{res}_{\alpha=\alpha_k}
    (P-i\varepsilon)^{-\alpha}(x,x)
    =
    \frac{i^\kappa}{(4\pi)^{n/2}\Gamma(\alpha_k)}
    \sum_{j=0}^{k}
    \frac{(i\varepsilon)^{k-j}}{(k-j)!}\,
    u_j(x,x) \in \mathrm{End}(E_x),
\end{equation}
and
\begin{equation} \label{resnolimplus}
    \operatorname*{res}_{\alpha=\alpha_k}
    (P+i\varepsilon)^{-\alpha}(x,x)
    =
    \frac{(-i)^\kappa}{(4\pi)^{n/2}\Gamma(\alpha_k)}
    \sum_{j=0}^{k}
    \frac{(-i\varepsilon)^{k-j}}{(k-j)!}\,
    u_j(x,x) \in \mathrm{End}(E_x).
\end{equation}
If \(\alpha_k\in -\mathbb N_0\), the apparent singularity cancels out.
\end{theorem}

\begin{proof}
As explained above, we are interested in the residues of  $\frac{1}{\Gamma(\alpha)} F^\kappa_{j + \alpha - 1}(i\epsilon,0)$, which are given in (\ref{diagmodel}), namely 
\begin{equation*} 
    \frac{i^\kappa 
    \Gamma\!\left(\alpha+j-\frac n2\right)}
    {(4\pi)^{n/2}\Gamma(\alpha)}
    (-i\varepsilon)^{\frac n2-\alpha-j}.
\end{equation*}
At \(\alpha=\alpha_k\), the factor $
    \Gamma\!\left(\alpha+j-\frac n2\right) = \Gamma(j-k+\alpha-\alpha_k)$
has residue $
    \frac{(-1)^{k-j}}{(k-j)!}$
for \(j\leq k\), and is regular for \(j>k\). Therefore, the above initial residue at \(\alpha=\alpha_k\) reads 
\begin{equation*} 
     \frac{i^\kappa (i\varepsilon)^{k-j} }
    {(4\pi)^{n/2}\Gamma(\alpha_k) (k-j)!}.
\end{equation*}
From (\ref{eqresid}), the residues of $\zeta^\kappa_{g,\varepsilon}$ are then given by 
\begin{equation*} 
    \operatorname*{res}_{\alpha=\alpha_k}
    (P-i\varepsilon)^{-\alpha}(x,x)
    =
    \frac{i^\kappa}{(4\pi)^{n/2}\Gamma(\alpha_k)}
    \sum_{j=0}^{k}
    \frac{(i\varepsilon)^{k-j}}{(k-j)!}\,
    u_j(x,x).
\end{equation*} 
Note again that the holomorphic remainder $\rho_N$ in Theorem \ref{thmexpanresol} does not contribute a residue. The calculation for the lower-half-plane applies accordingly.
\end{proof}

\begin{proposition} \label{propthmfinal}
For every \(k \in \N_0\) with $
    \frac n2-k\in\pazocal {P}_n$,
taking the residue first and then letting $\epsilon \to 0^+$ gives
\begin{equation} \label{respropfin-}
    \lim_{\epsilon \to 0^+}
    \operatorname*{res}_{\alpha=\frac n2-k}
    (P-i\varepsilon)^{-\alpha}(x,x)
    =
    \frac{i^\kappa u_k(x,x)}
    {(4\pi)^{n/2}\Gamma\!\left(\frac n2-k\right)} \in \mathrm{End}(E_x),
\end{equation}
and
\begin{equation} \label{respropfin+}
    \lim_{\epsilon \to 0^+}
    \operatorname*{res}_{\alpha=\frac n2-k}
    (P+i\varepsilon)^{-\alpha}(x,x)
    =
    \frac{(-i)^\kappa u_k(x,x)}
    {(4\pi)^{n/2}\Gamma\!\left(\frac n2-k\right)} \in \mathrm{End}(E_x),
\end{equation}
as limits in \(C^\infty\).
\end{proposition}

\begin{proof}
Consider the upper-half-plane case. In the residue formula (\ref{resnolimminus}) proven in Theorem \ref{thmresbeforelim},
\begin{equation*} 
    \operatorname*{res}_{\alpha=\alpha_k}
    (P-i\varepsilon)^{-\alpha}(x,x)
    =
    \frac{i^\kappa}{(4\pi)^{n/2}\Gamma(\alpha_k)}
    \sum_{j=0}^{k}
    \frac{(i\varepsilon)^{k-j}}{(k-j)!}\,
    u_j(x,x),
\end{equation*}
 every term where \(j<k\) in the sum contains a positive power
of \(\varepsilon\). Therefore, taking the limit makes the residue go to zero. Only the case \(j=k\) survives, and it yields $u_k(x,x)$ after taking the limit. 

Hence we get 
\begin{equation*}
    \lim_{\epsilon \to 0^+}
    \operatorname*{res}_{\alpha=\frac n2-k}
    (P-i\varepsilon)^{-\alpha}(x,x)
    =
    \frac{i^\kappa u_k(x,x)}
    {(4\pi)^{n/2}\Gamma\!\left(\frac n2-k\right)}.
\end{equation*}
The limit from the lower half-plane is analogous and yields (\ref{respropfin+}).
\end{proof}

\begin{proposition} The pseudo-Riemannian spectral zeta function density $\alpha \mapsto \zeta^\kappa_{g,\varepsilon}(\alpha)$ is holomorphic at $\alpha=0$ in every dimension. \end{proposition} \begin{proof} If $n$ is odd, then $\alpha=0$ does not belong to the pole set \[ \pazocal{P}_n = \left\{ \tfrac{n}{2},\, \tfrac{n}{2}-1,\, \tfrac{n}{2}-2,\, \ldots \right\}, \] which consists entirely of half-integers in this case. Hence $\zeta^\kappa_{g,\varepsilon}$ is holomorphic at $\alpha=0$. If $n$ is even, the only candidate pole at $\alpha=0$ corresponds to $k=n/2$. From~\eqref{diagmodel} and the discussion that follows, we see that this potential singularity is cancelled by the zero of $1/\Gamma(\alpha)$ at $\alpha=0$. 
\end{proof}

\subsection{The scalar curvature contribution} \label{sectsclarcurvcontrib} Finally, we explain how one can extract the scalar curvature from the residue of $\alpha \mapsto \zeta^\kappa_{g,\varepsilon}$ at $\alpha = \frac{n}{2}-1$.

Recall that the Laplace-type operator $P \in \mathrm{Diff}_{\mathrm{sc}}^{2,0} (M;E)$ admits a unique decomposition

\begin{equation} \label{laptypedecomp}
    P=\Box^{\nabla}+B,
\qquad
\Box^{\nabla}
:=
g^{ij}\left(
\nabla_i\nabla_j-\Gamma_{ij}^{k}\nabla_k
\right),
\end{equation}
where $\nabla$ is a scattering connection on $E$ and \(B\in C^\infty\left(\, \overbar{M};\operatorname{End}(E)\right)\). Here $g^{ij}$ denotes the inverse metric, considered as a section of ${}^{\mathrm{sc}}T \, \Mb \otimes_s {}^{\mathrm{sc}}T \, \Mb$, and \(\Gamma^k_{ij}\) are the Christoffel symbols of the Levi--Civita
connection of \(g\).

\begin{proposition} \label{propscalarcurvgeneral} Under the hypotheses that make Theorem \ref{thmexpanresol} valid, for $n \geq 3$,
    $$\lim_{\epsilon \to 0^+}
    \operatorname*{res}_{\alpha=\frac n2-1}
    (P \pm i\varepsilon)^{-\alpha}(x,x)
    =
    \frac{(\mp i)^\kappa}
    {(4\pi)^{n/2}\Gamma\!\left(\frac n2-1\right)} \left( -\frac{R_g(x)}{6} \operatorname{Id}_{E_x} - \, B(x) \right) \in \mathrm{End}(E_x),$$
    where $R_g(x)$ is the scalar curvature of $(M,g)$ at $x$.
\end{proposition}

\begin{proof}
    Once again, fix \(x_0\in M\), choose normal coordinates centered at \(x_0\), and
trivialize \(E\) by parallel transport along the radial geodesics
emanating from \(x_0\). From (\ref{urecursion}), the $j=0$ transport equation reads
$$2\nabla_V u_0 + \mathfrak h u_0 = 0, \quad u_0(x_0,x_0) = \operatorname{Id}_{E_{x_0}}.$$ 
The connection \(\nabla\) is the one appearing in the Laplace-type decomposition of \(P\), and it acts on the first variable of the Hadamard coefficients. The second transport equation yields
\[
\nabla_Vu_1+
\left(1+\frac{\mathfrak h}{2}\right)u_1
=
-Pu_0.
\]
In the radial parallel trivialization, the first equation gives
\[
u_0(x,x_0)
=
|g(x_0)|^{1/4}|g(x)|^{-1/4}
\operatorname{Id}_{E_{x_0}}.
\]
Since $ V(x_0)=0$ and $\mathfrak h(x_0)=0$, evaluation of the second transport equation at \(x=x_0\) yields
\[
u_1(x_0,x_0)
=
-\left.Pu_0(x,x_0)\right|_{x=x_0},
\]
where the operator $P$ acts on the first variable. 

Set now $ f(x):=|g(x_0)|^{1/4}|g(x)|^{-1/4}$. Writing the connection in the chosen frame as 
$\nabla_i=\partial_i+A_i$, 
the radial parallel condition implies $ x^iA_i(x)=0$.
Consequently, differentiating \(x^iA_i(x)=0\) once and twice at \(x_0\) gives
\[
A_i(x_0)=0,
\qquad
\partial_iA_j(x_0)+\partial_jA_i(x_0)=0,
\]
and hence, by symmetry of \(g^{ij}\),
\[
g^{ij}(x_0)\partial_iA_j(x_0)=0.
\]
Moreover, in normal coordinates and with our curvature convention,
\[
f(x)
=
1+\frac{1}{12}
\mathbf{Ric}_{ij}(x_0)x^ix^j
+\pazocal{O}(|x|^3),
\]
where $\mathbf{Ric}_{ij}$ is the Ricci tensor. This identity follows from \cite[Prop.~1.28]{344aea9f8a67a494d99261c648fbf0631} and holds in arbitrary signature. Now, since we have
\[
\Gamma_{ij}^{k}(x_0)=0,
\qquad
A_i(x_0)=0,
\qquad
\partial_i f(x_0)=0, \quad \textnormal{and} \ \ \ \partial_i\partial_j f(x_0)
=
\frac16 \mathbf{Ric}_{ij}(x_0),
\]
knowing that $R_g=g^{ij}\mathbf{Ric}_{ij}$ we obtain 
\begin{align*}
\left.
\Box^\nabla
\bigl(f(x)\operatorname{Id}_{E_{x_0}}\bigr)
\right|_{x=x_0}
&=
g^{ij}(x_0)
\left(
\partial_i\partial_j f(x_0)
\operatorname{Id}_{E_{x_0}}
+ \, 
\partial_iA_j(x_0)
\right) \\
&=
\frac{R_g(x_0)}{6}
\operatorname{Id}_{E_{x_0}}.
\end{align*}
Observe that it is the symmetry of \(g^{ij}\), combined with the antisymmetry of $\partial_iA_j(x_0)$, that kills the connection contribution. 

Since \(u_0(x_0,x_0)=\operatorname{Id}_{E_{x_0}}\) and
\(P=\Box^\nabla+B\), it follows that
\[
\left.Pu_0(x,x_0)\right|_{x=x_0}
=
\frac{R_g(x_0)}{6}\operatorname{Id}_{E_{x_0}} + \ B(x_0).
\]
Therefore,
\begin{equation} \label{u_1curv}
    u_1(x_0,x_0)
=
-\frac{R_g(x_0)}{6}\operatorname{Id}_{E_{x_0}}
- \ B(x_0),
\end{equation}
for any $x_0 \in M$. Proposition~\ref{propthmfinal} gives, for $k=1$,
\begin{equation*}
    \lim_{\epsilon \to 0^+}
    \operatorname*{res}_{\alpha=\frac n2-1}
    (P\pm i\varepsilon)^{-\alpha}(x,x)
    =
    \frac{(\mp i)^\kappa u_1(x,x)}
    {(4\pi)^{n/2}\Gamma\!\left(\frac n2-1\right)}.
\end{equation*}
Together with (\ref{u_1curv}), this yields the desired result.
\end{proof}
Note that Proposition \ref{propscalarcurvgeneral} therefore extends the contribution of the scalar curvature, from the case of the Laplace--Beltrami operator in \cite[Prop.~8.5]{Dang:2020jch}, to general Laplace-type operators.

\begin{remark} [Zeroth-order coefficient in radial parallel trivialization]
\textnormal{In a local trivialization of \(E\), write
\[
P
=
g^{ij}\partial_i\partial_j
+\mathfrak a^k\partial_k+ \mathfrak c,
\qquad
\nabla_i=\partial_i+A_i,
\]
where $\mathfrak a^k$ and $\mathfrak c$ are smooth sections of $\operatorname{End}(E)$.
Then the identity \(P=\Box^\nabla+B\) gives
\[
\mathfrak a^k
=
2g^{kj}A_j-g^{ij}\Gamma^k_{ij}\Id_E, \quad B
=
\mathfrak c-g^{ij}
\left(
\partial_iA_j+A_iA_j-\Gamma^k_{ij}A_k
\right). \]
In particular, in normal coordinates centered at \(x_0\) and in a radial
parallel trivialization, one has
\[
\Gamma^k_{ij}(x_0)=0,\qquad A_i(x_0)=0,
\qquad
g^{ij}(x_0)\partial_iA_j(x_0)=0,
\]
and hence $
B(x_0)= \mathfrak c(x_0)$. In other words the intrinsic zeroth-order term \(B(x_0)\) coincides with the zeroth-order coefficient \( \mathfrak c(x_0)\) of the local expression of \(P\).}

\end{remark}

\begin{remark}[The two-dimensional case]
\textnormal{When \(n=2\), the pole set reduces to
\[
    \pazocal{P}_2=\{1\}.
\]
In the case \(\kappa=1\), Proposition \ref{propthmfinal} gives
\begin{equation*}
    \lim_{\epsilon \to 0^+}
    \operatorname*{res}_{\alpha=1}
    (P-i\varepsilon)^{-\alpha}(x,x)
    =
    \frac{i}{4\pi} \mathrm{Id}_{E_x},
    \qquad
    \lim_{\epsilon \to 0^+}
    \operatorname*{res}_{\alpha=1}
    (P+i\varepsilon)^{-\alpha}(x,x)
    =
    -\frac{i}{4\pi} \mathrm{Id}_{E_x}.
\end{equation*}}\textnormal{The candidate pole at \(\alpha=0\), corresponding to \(k=1\), is removable.
Thus the coefficient giving rise to the scalar curvature does not occur as a residue in dimension
two.}

\textnormal{Indeed, the \(j=0\) and \(j=1\) terms in the upper-half-plane model term (\ref{diagmodel}) satisfy
\[
    \left.
    \frac{i^\kappa u_0(x,x)\Gamma(\alpha-1)}
    {4\pi\,\Gamma(\alpha)}
    (-i\varepsilon)^{1-\alpha}
    \right|_{\alpha=0}
    =
    \frac{i^{\kappa+1}\varepsilon}{4\pi}u_0(x,x)
    \to 0,
\]
whereas
\[
    \left.
    \frac{i^\kappa u_1(x,x)\Gamma(\alpha)}
    {4\pi\,\Gamma(\alpha)}
    (-i\varepsilon)^{-\alpha}
    \right|_{\alpha=0}
    =
    \frac{i^\kappa}{4\pi}u_1(x,x).
\]}\textnormal{Hence, in the mixed-signature case, the local contribution at
\(\alpha=0\) is 
   $ \frac{i}{4\pi}u_1(x,x)$.}

\textnormal{If we consider the unique decomposition $P=\Box^{\nabla}+B$, one has
\[
    u_1(x,x)=- \frac{R_g(x)}{6}\operatorname{Id}_{E_x} - \ B(x)
\]}\textnormal{by formula (\ref{u_1curv}) in the proof of Proposition \ref{propscalarcurvgeneral}, where $R_g(x)$ is the scalar curvature of $(M,g)$ at $x$. The holomorphic remainder in Theorem \ref{thmexpanresol} may still add a global contribution.}
\end{remark}

\smallskip

\section{applications of the residue formula} \label{sectapp}

In this final section, we apply the complex-power residue formula established in Proposition~\ref{propthmfinal} in two different settings. We first consider the wave operator on a Lorentzian manifold and treat the odd-dimensional case, complementing Dang--Wrochna's even-dimensional results in \cite{Dang:2020jch}. We then turn to the Dirac operator on scattering Clifford modules and express the residues of its complex powers, extending Dang--Vasy--Wrochna's formula in \cite[Thm.~5.7]{dang2024diracoperatorslocalinvariants} to arbitrary dimension and signature.

\subsection{The odd-dimensional Lorentzian case} \label{sectLorentz}

In this section, we specialize to $\kappa = 1$, $n$ odd and $E$ trivial. Let $(M,g$) be a non-trapping Lorentzian scattering space of odd dimension $n$, in the sense of Definition \ref{scpseudospace}. Denote by $\Box_g$ the \textit{wave operator} on $M$, which is in local coordinates $$\Box_g = \frac{1}{\sqrt{\vert g \vert}} \partial_{\mu}\left(\sqrt{\vert g \vert} g^{\mu \nu} \partial_{\nu} \right).$$ Namely it is the Laplace--Beltrami operator on $M$ in Lorentzian signature, where $\vert g \vert := \vert  \mathrm{det}(g) \vert $. We now consider $P := \Box_g + V$, where $V \in C^{\infty}(\, \Mb)$ is real-valued, and denote by $u_k$ its Hadamard coefficients in the convention of Section \ref{transeq}. Note that, since \(V\) is real-valued, \(P=\Box_g+V\) is formally self-adjoint and has the same principal symbol, hence the same non-trapping Hamilton flow, as \(\Box_g\). Therefore \(P\) satisfies Hypothesis \ref{theH}. For \(V=0\), the wave operator is moreover essentially self-adjoint by Vasy~\cite{Vasy2017EssentialSO} under the non-trapping assumption.

\begin{theorem} \label{thmresidlorentz}
    Assume $(M,g)$ is a globally hyperbolic, non-trapping Lorentzian scattering space in the sense of Definition \ref{Pnontrap}. Then, for all $\epsilon > 0$, the Schwartz kernel of $(P \pm i\epsilon)^{-\alpha}$ has for $\Re(\alpha) > \frac{n}{2}$ a well-defined on-diagonal restriction $\iota_\Delta^*(P \pm i\varepsilon)^{-\alpha} \in \pazocal{D}'(M)$, which extends as a meromorphic function of $\alpha \in \C$ with poles in $\pazocal{P}_n := \left\{\frac{n}{2}, \frac{n}{2} - 1, \ldots, \frac{1}{2}, -\frac{1}{2}, \ldots   \right\}$. Its residues in $\alpha \in \pazocal{P}_n$ are given by 
    $$ \lim_{\epsilon \to 0^+} \operatorname*{res}_{\alpha=\frac n2-k}(P \pm i\epsilon)^{-\alpha}(x,x) = \mp \frac{i\, u_k(x, x)}{(4\pi)^{\frac{n}{2}} \Gamma(\frac{n}{2} - k)},$$
for any $k \in \N_0$, where $u_k$ are the Hadamard coefficients of $P$. 
\end{theorem}

\begin{proof}
First, note that in \cite{Dang:2020jch}, parity only becomes relevant when the residues are computed in \cite[\S~8]{Dang:2020jch}. One could repeat these calculations in the odd-dimensional case and get the desired residues directly.

Alternatively, the formula follows directly from our pseudo-Riemannian result. Indeed, global hyperbolicity implies strong causality, as highlighted in our discussion at the beginning of Section \ref{returnSect}. Hence the causally convex normal neighborhoods provide relatively forward convex cut-off pairs. Let then $(\U_0, \chi)$ be a relatively forward convex pair in the sense of Definition \ref{defrelforconvdef}. Hence we are able to apply Theorem \ref{thmexpanresol} (1), which constitutes the smooth local expansion proven by Dang–Wrochna in \cite[Lem.~7.1]{Dang:2020jch} for the complex powers of $P$. More precisely, we obtain for any given $s,p$ that
\begin{equation*} 
            \Pa = \frac{1}{\Gamma(\alpha)} \sum_{k=0}^{N} u_k \mathbf{F}_{k + \alpha - 1}(i\epsilon) + R_N(i\epsilon,\alpha)
        \end{equation*}
as distributions on $\U_0$, where $R_N(i\epsilon,\alpha)$ is in $C^s(\U_0)$ and holomorphic for $\Re(\alpha) > -p$.

Finally, by Proposition \ref{propthmfinal}, we obtain 
\begin{equation*}
    \lim_{\epsilon \to 0^+}
    \operatorname*{res}_{\alpha=\frac n2-k}
    (P\pm i\varepsilon)^{-\alpha}(x,x)
    =
    \mp \frac{i u_k(x,x)}
    {(4\pi)^{n/2}\Gamma\!\left(\frac n2-k\right)},
\end{equation*}
for $\alpha$ running in $$\pazocal{P}_n = \left\{ \tfrac{n}{2},\, \tfrac{n}{2}-1,\, \tfrac{n}{2}-2,\, \ldots \right\},$$ which gives the odd-dimensional analogue of \cite[Thm.~8.3]{Dang:2020jch} for $V = 0$.
\end{proof}
Note that global hyperbolicity is not required to apply Theorem~\ref{thmexpanresol} under the relative forward convexity assumption; strong causality already suffices, as observed in Remark~\ref{remstronglycausal}.
In particular, again from formula (\ref{u_1curv}) in the proof of Proposition \ref{propscalarcurvgeneral}, we obtain the extraction of the scalar curvature from the Hadamard coefficient $u_1$. Namely, we have $$u_1(x,x) = - \frac{R_g(x)}{6} - V(x),$$ where $R_g(x)$ is the scalar curvature of $(M,g)$ at $x$. 

\begin{proposition} [Scalar curvature in odd dimension]
\label{proprescurvlorentz}

As a particular case of Theorem \ref{thmresidlorentz}, for $n \geq 3$, one obtains
   $$ \lim_{\epsilon \to 0^+}  \operatorname*{res}_{\alpha=\frac n2-1}(P \pm i\epsilon)^{-\alpha}(x,x) = \frac{\pm i}{(4\pi)^{\frac{n}{2}} \Gamma\left(\frac{n}{2} -1 \right)}\left(\frac{R_g(x)}{6} + V(x)\right).$$
\end{proposition}
For $V = 0$, this is the direct analogue of \cite[Prop.~8.5]{Dang:2020jch}.

\subsection{Residue formula for the pseudo-Riemannian Dirac operator} \label{sectresdirac}

In \cite{dang2024diracoperatorslocalinvariants}, the authors establish the corresponding residue formula for the square of the Dirac operator acting on scattering spinors, or more generally on sections of a Clifford module, over an even-dimensional Lorentzian manifold; see \cite[Thm.~5.7]{dang2024diracoperatorslocalinvariants}.

We will then introduce the notions considered in their result, adapt them to the present setting, and then state the corresponding residue formula for the Dirac operator in arbitrary pseudo-Riemannian signature and dimension. Note that we will also discuss the choice of the inner product with respect to which the formal adjoint of $P$ is defined, allowing us to apply our previous results. In the Lorentzian case, such a positive-definite structure is obtained from the choice of a timelike vector field, available in particular in the globally hyperbolic setting, whereas in higher signature cases the precise relevant notion is that of \emph{spacelike reflection}.

Recent developments in twisted noncommutative geometry actually provide an
algebraic mechanism for the emergence of pseudo-Riemannian signature. In particular, Nieuviarts \cite{nieuviarts2026emergencetimetwistedspectral} relates twisted and pseudo-Riemannian
spectral triples through a \(K\)-morphism associated with a spacelike
reflection, while Martinetti \cite{martinetti2026twistedstandardmodelkrein} shows that the inner product induced by a broad
class of twists naturally defines a \emph{Krein} structure.
These approaches are complementary to the present work. They concern
the algebraic origin of the indefinite spinorial structure, whereas
we start from a pseudo-Riemannian Clifford module and
compute the local residues of the Feynman complex powers of $P=-\slashed D^{\,2}$. Thus, formula~(\ref{resDiracfinal}) in Theorem \ref{thmdiracresidue} provides, in this framework, local geometric coefficients which are relevant to a pseudo-Riemannian spectral action.

\subsubsection{The scattering Dirac operator} \label{sectDiracsc}

Let $(M,g)$ be a pseudo-Riemannian scattering space of dimension $n$, with signature $(\kappa,n-\kappa)$. Consider the Levi-Civita connection $\nabla^{\mathrm{L}}$ on $(\, \Mb,g)$. It is a \emph{scattering connection}, in particular
$$\nabla^{\mathrm{L}} : C^\infty\!\left(\, \overbar M;{}^{\mathrm{sc}}T \, \overbar M\right)
\to
C^\infty\!\left(\, \Mb ; \overbar
{}^{\mathrm{sc}}T^*\, \overbar M
\otimes{}^{\mathrm{sc}}T \, \overbar M
\right),$$
and it is a lift of a \emph{b-connection}. See \cite[\S~1.2]{Kottke2012ACI} for the complete definition of a scattering connection, and see the classical references \cite{Lawson:1998yr,344aea9f8a67a494d99261c648fbf0631} for Dirac operators on Clifford modules. 

The Levi-Civita connection $\nabla^{\mathrm{L}}$ induces a connection on $C({}^{\mathrm{sc}}T \, \Mb)$, the \emph{scattering Clifford bundle}, which is defined in the same way as $C(TM)$, the Clifford bundle of $(M,g)$, but with respect to the scattering tangent bundle ${}^{\mathrm{sc}}T \, \Mb$ instead of the usual tangent bundle $TM$. This connection will still be denoted by $\nabla^{\mathrm{L}}$. Now, consider a \emph{Clifford module} $(E, \nabla^E)$ that is \emph{scattering} in the sense that it is endowed with a linear map
$$\gamma : C^{\infty}\left(\, \Mb;C({}^{\mathrm{sc}}T \, \Mb)\right) \rightarrow C^{\infty}\left(\, \Mb; \mathrm{End}(E)\right),$$
called the \emph{Clifford action}, satisfying the \emph{Clifford relations}
$$\gamma(X)\gamma(Y) + \gamma(Y)\gamma(X) = -2 g(X,Y) \, \mathrm{Id}_{E},$$
for $X,Y \in \pazocal{V}_{\text{sc}}(\, \overbar{M})$. It is equipped with a \emph{Clifford connection} $\nabla_E$, which by definition satisfies, for all $X \in \pazocal{V}_{\text{sc}}(\, \overbar{M})$, $Y \in C^{\infty}\left(\,\Mb;C({}^{\mathrm{sc}}T \, \Mb)\right)$ and $u \in C^{\infty}(\,\Mb;E)$,
$$\nabla^E_X (\gamma(Y)u) = \gamma(\nabla^{\mathrm{L}}_XY )u + \gamma(Y )\nabla^E_X u. $$

The \emph{scattering Dirac operator} $\slashed{D} \in \mathrm{Diff}^{1,0}_{\mathrm{sc}}(M;E)$ is then given by the composition
\begin{equation} \label{defdirac}
    C^{\infty}(\,\Mb;E) \overset{\nabla^E}{\longrightarrow} C^{\infty}(\,\Mb;E \otimes {}^{\mathrm{sc}}T^* \, \Mb ) \underset{g}{\overset{\sim}{\longrightarrow}} C^{\infty}(\,\Mb;E \otimes {}^{\mathrm{sc}}T\, \Mb ) \overset{\gamma}{\longrightarrow} C^{\infty}(\,\Mb;E).
\end{equation}
In a local orthonormal frame $(e_1, \ldots, e_n)$ of ${}^{\mathrm{sc}}T \, \Mb$, it reads 
\begin{equation*} 
   \slashed{D} = g^{ab} \gamma(e_a) \nabla^E_{e_b}, 
\end{equation*}
where $g^{ab}$ denotes the inverse metric. We put $$P := - \slashed{D}^2 \in \mathrm{Diff}^{2,0}_{\mathrm{sc}}(M;E),$$
which is a differential operator of Laplace-type.

We assume the existence of a smooth positive-definite Hermitian bundle metric $h$, namely $h \in C^{\infty}\left(\,\Mb; \mathrm{Hom}(E,E^*)\right)$ such that $h_x$ is positive-definite for each $x \in \, \Mb$. Therefore, for any compactly supported sections $u,v$ of $E$, we define the positive scalar product $\langle \cdot,\cdot \rangle_E$ by integrating $h(u,v)$ against the volume density $d\mathrm{vol}_g$, i.e.
\begin{equation*} 
    \langle u,v \rangle_E = \int_M h_x(u_x,v_x) \, d\mathrm{vol}_g.
\end{equation*}
The formal adjoint $P^*$ of $P$ is thus defined with respect to the inner product $\langle \cdot,\cdot \rangle_E$. 

However, we still assume the existence of a section $\beta \in C^\infty\left(\, \Mb;\operatorname{Hom}(E,E^*)\right)$, possibly indefinite, such that $\beta_x$ is Hermitian non-degenerate for each $x \in \, \Mb$, together with
\begin{enumerate} \label{betaform}
    \item $\beta\bigl(\gamma(X)u,v\bigr)
=
-\beta\bigl(u,\gamma(X)v\bigr),$ which is compatibility with Clifford multiplication,
 \item $X\bigl(\beta(u,v)\bigr)
=
\beta\bigl(\nabla_X^E u,v\bigr)
+
\beta\bigl(u,\nabla_X^E v\bigr),$ which is equivalent to $\nabla^E \beta = 0$.
\end{enumerate}
Again, we define the Hermitian form $\langle \cdot,\cdot \rangle_{\beta}$, by integrating $u,v \mapsto \beta (u,v)$ with the volume density $d \mathrm{vol}_g$ of $(M,g)$. 

\begin{remark} [Hermitian metric in the Lorentzian case]\label{remlorentzscprod}
\textnormal{In \cite[\S~5.1]{dang2024diracoperatorslocalinvariants}, the authors consider the Dirac operator acting on sections of the spin bundle $S$ and assume there exists a section $\beta \in C^{\infty}\left(\,\Mb; \mathrm{Hom}(S,S^*)\right)$ such that $\beta_x$ is Hermitian non-degenerate for each $x \in \, \Mb$. That pairing is generally indefinite in Lorentzian signature. Hence they assume the existence of a  scattering vector field $e \in \pazocal{V}_\mathrm{sc}(\, \Mb)$ that is timelike and future-directed in $M$ to produce a positive-definite metric from $\beta$; see condition b) in \cite[\S~5.1]{dang2024diracoperatorslocalinvariants}. That being said, assumptions 1. and 2. concerning $\beta$ are precisely assumptions~a) and~c) of \cite[\S~5.1]{dang2024diracoperatorslocalinvariants} for their $\beta$ on the spin bundle $S$.
Afterwards, they consider Clifford modules $E$ of the form $E = S \otimes H$, where $H$ is a Hermitian vector bundle. The definition of the initial form $\beta$ is extended to the twisted module and the positivity property is preserved. In the present setting, its role is played by $h$ on general Clifford modules. Their construction is thus based on notions particular to  Lorentzian geometry.}
\end{remark}

In summary, we assume the existence of two independent forms $\beta$ and $h$ as described above, and we consider the quadruple $(E,\nabla^E,\beta,h)$ as the given data for our Clifford module. Note that, strictly speaking, the data should contain the sc-metric $g$ and the Clifford action $\gamma$ as well.
\begin{definition} \label{Dnontrap}
    We say that $\slashed{D} \in \mathrm{Diff}^{1,0}_{\mathrm{sc}}(M;E)$ is \emph{non-trapping} if $\slashed{D}$ is the Dirac operator given by (\ref{defdirac}), and if $P = - \slashed{D}^2$ satisfies Hypothesis \ref{theH}.
\end{definition}

\begin{remark} [Terminology]
    \textnormal{Note that the \emph{non-trapping} Dirac operator of Definition \ref{Dnontrap} satisfies the same properties as its Lorentzian analogue, the Dirac operator on a \emph{small perturbation of (generalized) Minkowski space} in \cite[Def. 5.3]{dang2024diracoperatorslocalinvariants}; we simply chose the shorter appellation \emph{non-trapping} in our case.}
\end{remark}

\subsubsection{A construction of a positive-definite metric from the indefinite form $\beta$} In the Lorentzian case, as mentioned in Remark \ref{remlorentzscprod}, the authors obtain a positive-definite form out of the original indefinite form $\beta$ on the spin bundle $S$, using a scattering vector field $e \in \pazocal{V}_\mathrm{sc}(\, \Mb)$. As highlighted in \cite[Rem. 5.1]{dang2024diracoperatorslocalinvariants}, if the Lorentzian manifold is assumed to be globally hyperbolic, a chosen Cauchy foliation supplies such a unit timelike normal which converts \(\beta\) into a positive-definite Hermitian metric. A given Cauchy foliation gives a canonical choice for such a timelike $e$; the authors must still impose condition b) in \cite[\S~5.1]{dang2024diracoperatorslocalinvariants}, namely $i \beta \gamma(e) > 0$.

In the following, we give an analogous construction in the pseudo-Riemannian setting, namely we obtain a positive-definite Hermitian metric from the indefinite form $\beta$ defined on a Clifford module $E$ satisfying assumptions 1. and 2. above.

Assume that we are given an orthogonal direct sum decomposition of the scattering tangent bundle
\begin{equation} \label{tangentsplit}
    {}^{\mathrm{sc}}T \, \overbar M
=
T^+ \oplus T^-,
\qquad
\operatorname{rk}(T^+)=\kappa,
\end{equation}
 where $g$ is negative-definite on $T^-$ and positive-definite on $T^+$. Note that such a splitting always exists on a pseudo-Riemannian manifold; here, we simply fix one.
 
Assume now that \(T^+\) is oriented. The associated \emph{spacelike
reflection} is the bundle endomorphism $ r:{}^{\mathrm{sc}}T \, \overbar M
\to {}^{\mathrm{sc}}T \, \overbar M$
defined by
\[
r|_{T^-}=-\operatorname{Id},
\qquad
r|_{T^+}=\operatorname{Id}.
\]
It satisfies
\[
r^2=\operatorname{Id},
\qquad
g(rX,rY)=g(X,Y),
\]
and the metric $$g_r(X,Y):=g(X,rY)$$ is positive-definite. We refer to \cite[\S~3]{van_den_Dungen_2013} for more details on such notions and we recall that these are related to the construction of spectral triples in noncommutative geometry \cite{Strohmaier_2006,martinetti2026twistedstandardmodelkrein, nieuviarts2026emergencetimetwistedspectral}.

Let \((e_1,\ldots,e_{\kappa})\) be a local oriented orthonormal frame
of \(T^+\). We define
\begin{equation*} 
    J_+
=
i^{\frac{\kappa(\kappa +1)}{2}}\,\gamma(e_1)\cdots\gamma(e_\kappa),
\end{equation*}
for which one can show that
\[
J_+^2=\operatorname{Id},
\qquad
\beta(J_+u,v)=\beta(u,J_+v).
\]
 If, in addition, we assume
\[
\beta(u,J_+u)>0
\qquad
\text{for every }u\neq0,
\]
then $J_+$ is a \emph{fundamental symmetry}, and
\begin{equation*} 
    h(u,v):=\beta(u,J_+v)
\end{equation*}
is a smooth positive-definite Hermitian metric on \(E\). Note that the definition of $J_+$ depends on the choice of the decomposition (\ref{tangentsplit}).

\begin{remark} [Fundamental symmetry in Lorentzian signature]
    \textnormal{The Lorentzian case corresponds in our convention to \(\kappa=1\), and thus \(T^+\) is locally generated by a unit timelike vector field \(e\). Hence the fundamental symmetry is given by $ J_+ = i \gamma(e)$. The condition $\beta(u,J_+u)>0$ imposed above is precisely the analogue of assumption b) in \cite[\S~5.1]{dang2024diracoperatorslocalinvariants}, namely $i \beta \gamma(e) > 0$.}
    
\textnormal{Thus, strictly speaking, a Cauchy foliation is not required for the construction itself:
in the globally hyperbolic Lorentzian case, it merely provides a natural choice of the timelike line bundle \(T^+\) through its unit normal, as  highlighted in \cite[Rem. 5.1]{dang2024diracoperatorslocalinvariants}. Therefore, the construction based on such an orthogonal decomposition of the tangent bundle (or equivalently on a spacelike reflection) directly extends the Lorentzian setting.}

\end{remark}
\begin{remark} [Self-adjointness of the Dirac operator in pseudo-Riemannian signature]\textnormal{ If the positivity condition $
\beta(u,J_+u)>0$ is not imposed, the formula $
h(u,v)=\beta(u,J_+v)$
still defines a Hermitian form, which may be indefinite. One may then study the Dirac operator with respect to the corresponding indefinite inner product, in the framework of \emph{Krein spaces}. Under suitable geometric assumptions, the pseudo-Riemannian Dirac operator has been shown to be essentially \emph{Krein self-adjoint}; see for instance \cite[Prop.~3.4]{van_den_Dungen_2013} for the Dirac operator acting on spinors.}
\end{remark}

\subsubsection{The residues of the complex powers of the square of the Dirac operator} 

Consider a Clifford module $(E,\nabla^E)$ over $\, \Mb$. We denote by
$$
F^E \in \Omega^2 \bigl(\, \Mb;\operatorname{End}(E)\bigr)
$$
the \emph{twisting curvature} of the Clifford module $E$, which commutes with the Clifford action. In signature \((\kappa,n-\kappa)\), put
\[
g(e_i,e_j)=\varepsilon_i\delta_{ij},
\qquad
\varepsilon_i\in\{\pm1\},
\]
where $(e_1, \ldots, e_n)$ is a local orthonormal frame  of ${}^{\mathrm{sc}}T \, \Mb$. The  Clifford contraction of $F^E$, denoted by $\mathbf{F}^E$, then reads 
$$ \mathbf F^E =
\sum_{i<j} \varepsilon_i \varepsilon_j 
\gamma(e_i)\gamma(e_j)F^E(e_i,e_j)
\in C^\infty \left(\, \overbar{M};\operatorname{End}(E)\right).$$
 By the Bochner--Lichnerowicz formula,
\begin{equation} \label{bochlichformula}
    \slashed D^{\,2}
=
(\nabla^E)^{*}\nabla^E
+
\mathbf F^E
+
\frac14R_g\operatorname{Id}_E,
\end{equation}
which holds in arbitrary signature; see \cite[Thm.~3.52]{344aea9f8a67a494d99261c648fbf0631} and \cite[Thm.~8.8]{Lawson:1998yr}. 

We emphasize that the formal adjoint $(\nabla^{E})^*$ of the Clifford connection is taken with respect to the Hermitian form $\langle\cdot,\cdot\rangle_\beta$, as explicitly noted in \cite[\S~5.2]{dang2024diracoperatorslocalinvariants}. For the Clifford contraction of the twisting curvature and the Bochner--Lichnerowicz formula in general pseudo-Riemannian signature, see \cite[\S\S~57--58]{Walpuski2022}.

Recall that, again by an abuse of notation, the fiber trace of the residue  \[
    \left[
    \operatorname{tr}_E
    \left(
        \underset{\alpha=\frac n2-1}{\operatorname{res}}
        \zeta^\kappa_{g,\varepsilon}(\alpha)
    \right)
    \right](x)
\]
will be denoted by $$
    \underset{\alpha=\frac n2-1}{\operatorname{res}}
    \operatorname{tr}_E\left((P-i\varepsilon)^{-\alpha}\right)(x,x)$$
in the following theorem.

\bigskip

\begin{theorem} \label{thmdiracresidue}
    Let $(M,g)$ be a pseudo-Riemannian scattering space of dimension $n$ and signature $(\kappa,n-\kappa)$. Assume the Dirac operator $\slashed{D} \in \mathrm{Diff}^{1,0}_{\mathrm{sc}}(M;E)$ is non-trapping as in Definition \ref{Dnontrap}, and put $ P=-\slashed{D}^{\,2}$. Then, for all $\varepsilon>0$, the Schwartz kernel of
$(P \pm i\varepsilon)^{-\alpha}$ has, for $
    \operatorname{Re}(\alpha)>\frac{n}{2}$,
a well-defined on-diagonal restriction $$
    \iota_\Delta^*(P \pm i\varepsilon)^{-\alpha} \in \pazocal{D}'\left(M;\mathrm{End}(E)\right),$$
which extends as a meromorphic function with
poles at $\pazocal{P}_n := \left\{ \frac{n}{2} - k \ \vert \ k \in \N_0, \ \frac{n}{2} - k \notin - \N_0 \right\}.$

Furthermore, for $n \geq 3$,
\begin{equation} \label{resDiracfinal}
    \lim_{\varepsilon\to0^+}
\underset{\alpha=\frac{n}{2}-1}{\operatorname{res}}
    (P \pm i\varepsilon)^{-\alpha}(x,x)
= \frac{(\mp i)^{\kappa}R_g(x)}
     {12(4\pi)^{\frac n2}
      \Gamma\!\left(\frac n2-1\right)} \mathrm{Id}_{{E_x}} + \frac{(\mp i)^{\kappa}\mathbf{F}^E(x)}
     {(4\pi)^{\frac n2}
      \Gamma\!\left(\frac n2-1\right)},
\end{equation}
where $\mathbf{F}^E$ denotes the Clifford contraction of the twisting curvature $F^E$ of $E$. In particular,
\begin{equation*} \label{resDirac}
    \lim_{\varepsilon\to0^+}
\underset{\alpha=\frac{n}{2}-1}{\operatorname{res}}
\operatorname{tr}_E
\left(
    (P \pm i\varepsilon)^{-\alpha}
\right)(x,x)
= \frac{(\mp i)^{\kappa}\operatorname{rk}(E)R_g(x)}
     {12(4\pi)^{\frac n2}
      \Gamma\!\left(\frac n2-1\right)},
\end{equation*}
where $\operatorname{rk}(E)$ is the rank of $E\to \, \Mb$ and $\tr_E$ is the fiber trace on $\mathrm{End}(E)$.
\end{theorem}

\begin{proof}
We apply Theorem~\ref{thmexpanresol}, since $P=-\slashed{D}^2$ satisfies Hypothesis \ref{theH}. Therefore, the on-diagonal restriction
 $$
    \iota_\Delta^*(P-i\varepsilon)^{-\alpha} \in \pazocal{D}'\left(M;\mathrm{End}(E)\right)$$
    is well-defined. Proposition~\ref{propthmfinal} then yields
\begin{equation} \label{respropfin-Dirac}
    \lim_{\epsilon \to 0^+}
    \operatorname*{res}_{\alpha=\frac n2-k}
    (P-i\varepsilon)^{-\alpha}(x,x)
    =
    \frac{i^\kappa u_k(x,x)}
    {(4\pi)^{n/2}\Gamma\!\left(\frac n2-k\right)} \in \mathrm{End}(E_x),
\end{equation}
for $k \in \N_0$ such that $ \frac n2-k \in  \pazocal{P}_n = \left\{ \frac{n}{2} - k \ \vert \ k \in \N_0, \ \frac{n}{2} - k \notin - \N_0 \right\}$.
Finally, the first Hadamard coefficient of $P$ can be expressed in terms of the scalar curvature $R_g$ and the Clifford contraction $\mathbf{F}^E$. Recall that the operator $P = - \slashed{D}^2$ is of Laplace-type, and set
\[
\Box^{\nabla^E}
:=
-(\nabla^E)^{*}\nabla^E.
\]
This coincides with the Laplace-type decomposition (\ref{laptypedecomp}) by integration by parts with respect to the pairing induced by \(\beta\), since we imposed $\nabla^E \beta = 0$. Observe now that from the Bochner--Lichnerowicz formula (\ref{bochlichformula}), $P = - \slashed{D}^2$ can be written as
\[
P = \Box^{\nabla^E}+B,
\qquad
B
=
-\mathbf F^E-\frac{R_g}{4}\operatorname{Id}_E.
\]
It follows from Proposition~\ref{propscalarcurvgeneral}
that
\[
u_1(x,x)
=
-\frac{R_g(x)}{6}\operatorname{Id}_{E_x} - \ B(x)
=
\frac{R_g(x)}{12}\operatorname{Id}_{E_x}
+ \ \mathbf F^E(x),
\]
or from \cite[Proof of Thm.~5.7]{dang2024diracoperatorslocalinvariants}, whose argument is independent of the signature. Together with (\ref{respropfin-Dirac}), one obtains
\begin{equation*}
    \lim_{\varepsilon\to0^+}
\underset{\alpha=\frac{n}{2}-1}{\operatorname{res}}
    (P-i\varepsilon)^{-\alpha}(x,x)
= \frac{i^{\kappa}R_g(x)}
     {12(4\pi)^{\frac n2}
      \Gamma\!\left(\frac n2-1\right)} \mathrm{Id}_{{E_x}} + \frac{i^{\kappa}\mathbf{F}^E(x)}
     {(4\pi)^{\frac n2}
      \Gamma\!\left(\frac n2-1\right)}.
\end{equation*}
Now, since the twisting curvature commutes with the Clifford action, each term $$\gamma(e_i)\gamma(e_j)F^E(e_i,e_j), \quad i\neq j$$ has vanishing trace. Hence $\mathrm{tr}_E(\mathbf{F}^E) = 0$. Therefore, applying the trace yields
\begin{equation*} 
    \lim_{\varepsilon\to0^+}
\underset{\alpha=\frac{n}{2}-1}{\operatorname{res}}
\operatorname{tr}_E
\left(
    (P-i\varepsilon)^{-\alpha}
\right)(x,x)
= \frac{i^{\kappa}\operatorname{rk}(E)R_g(x)}
     {12(4\pi)^{\frac n2}
      \Gamma\!\left(\frac n2-1\right)}.
\end{equation*}
The calculation in the lower-half-plane proceeds analogously. 
\end{proof}

\emph{Acknowledgments.} The author would like to thank his PhD thesis supervisor, Michał Wrochna, for valuable discussions.

\bibliographystyle{plainurl}
\bibliography{mybibi}

{\small
\noindent
{\scshape Mathematical Institute, Utrecht University, Utrecht, The Netherlands}\\
\emph{Email address:} \texttt{a.sadatkyaee@uu.nl}
}

\end{document}